\documentclass[a4paper]{article}

\usepackage{geometry}
\usepackage[T1]{fontenc}

\usepackage{amsmath, amssymb, amsthm}
\usepackage{mathtools}

\usepackage{hyperref}
\usepackage[dvipsnames]{xcolor}
\hypersetup{
colorlinks=true,
allcolors=MidnightBlue,
linktoc=page
}

\usepackage{authblk}

\usepackage{orcidlink}

\usepackage[ruled,vlined]{algorithm2e}
\usepackage{caption,subcaption}
\usepackage{booktabs}

\usepackage[capitalise]{cleveref}
\crefname{equation}{}{}
\Crefname{equation}{}{}
\creflabelformat{equation}{(#2#1#3)}
\crefname{algocf}{Algorithm}{Algorithms}
\crefname{figure}{Figure}{Figures}

\numberwithin{equation}{section}

\theoremstyle{plain}
\newtheorem{thm}{Theorem}[section]
\newtheorem{lem}[thm]{Lemma}

\newtheorem{prop}[thm]{Proposition}

\theoremstyle{remark}
\newtheorem{rem}{Remark}[section]

\theoremstyle{definition}
\newtheorem{defn}[thm]{Definition}

\crefname{thm}{Theorem}{Theorems}
\crefname{lem}{Lemma}{Lemmas}
\crefname{exm}{Example}{Examples}

\DeclareMathOperator*{\esssup}{ess\,sup}
\DeclareMathOperator{\supp}{supp}

\DeclareMathOperator{\Expect}{\mathbb E}

\DeclareMathOperator{\Law}{Law}
\DeclareMathOperator{\Ent}{Ent}
\DeclareMathOperator{\Var}{Var}

\DeclareMathOperator{\Id}{Id}

\newcommand{\proofstep}[1]{\medbreak{\noindent\itshape{#1}.}}

\title{Uniform-in-time propagation of chaos\\
for kinetic mean field Langevin dynamics%
\thanks{
The second named author's research is supported by
NSFC under the project No.~12371473.
The third named author's research is supported by
Finance For Energy Market Research Centre.}\
\thanks{The authors would like to thank an anonymous reviewer
for pointing out an error in a previous version of the paper.}
}

\author[1]{Fan Chen\,\orcidlink{0000-0003-0082-7908}}
\author[2]{Yiqing Lin}
\author[3]{Zhenjie Ren\,\orcidlink{0000-0003-4656-4074}}
\author[4]{Songbo Wang\,\orcidlink{0009-0009-3214-3587}}
\affil[1]{School of Mathematical Sciences, Shanghai Jiao Tong University,
Shanghai, China}
\affil[2]{MOE-LSC, Shanghai Jiao Tong University, Shanghai, China}
\affil[3]{CEREMADE, Université Paris-Dauphine, PSL, Paris, France}
\affil[4]{CMAP, École polytechnique, IP Paris, Palaiseau, France}

\begin{document}

\maketitle

\begin{abstract}
We study the kinetic mean field Langevin dynamics
under the functional convexity assumption of the mean field energy functional.
Using hypocoercivity,
we first establish the exponential convergence of the mean field dynamics
and then show the corresponding \(N\)-particle system converges
exponentially in a rate uniform in \(N\) modulo a small error.
Finally we study the short-time regularization effects of the dynamics
and prove its uniform-in-time propagation of chaos property
in both the Wasserstein and entropic sense.
Our results can be applied to the training of
two-layer neural networks with momentum
and we include the numerical experiments.
\end{abstract}

\bigskip

\noindent\textit{\href{https://mathscinet.ams.org/mathscinet/msc/msc2020.html}{MSC2020 Subject Classifications}:}
60J60, 60K35 (primary) 35B40, 35H10, 35Q83, 35Q84 (secondary)

\bigskip


\bigskip

\tableofcontents

\section{Introduction}
\label{sec:intro}

Training neural networks by momentum gradient descent
has proven to be effective in various applications
\cite{sutskever2013importance,kingma2014adam,ruder2016overview}.
However, despite their excellent performance,
the theoretical understanding of those algorithms remains elusive.
Recently, extensive researches have been conducted to model
the loss minimization of neural networks as a mean field optimization problem
\cite{mei2018mean,chizat2018global,rotskoff2018neural,HRSS19},
with most characterizing gradient descent algorithms
as \emph{overdamped mean field Langevin} (MFL) dynamics.
In this paper, we will focus on \emph{kinetic} dynamics instead,
which corresponds to momentum gradient descent
in the context of machine learning \cite{polyak1964some,kovachki2021continuous}.
Classical studies, such as \cite{hypocoer,ma2021there},
have explored the exponential convergence of linear kinetic Langevin dynamics
based on hypocoercivity and functional inequalities.
The kinetic MFL dynamics is studied in \cite{kazeykina2020ergodicity}
to model the momentum gradient descent for the training of neural networks
and its convergence to the unique invariant measure
is proven without a quantitative rate.
The present work studies both the quantitative long-time behavior
of the kinetic MFL dynamics
and its uniform-in-time \emph{propagation of chaos} (POC) property,
under a \emph{functional convexity} assumption,
and we aim to provide a theoretical justification for
the momentum algorithm's efficiency in practice.

\subsection{Settings and main results}
\label{sec:preview}

We give an informal preview of our settings and main results in this section.
Let \(F:\mathcal P_2(\mathbb R^d) \to \mathbb R\)
be a \emph{mean field functional}
and denote by
\(D_m F : \mathcal P_2(\mathbb R^d) \times \mathbb R^d \to \mathbb R^d\)
its \emph{intrinsic derivative}.
We aim to investigate the long-time behavior of the kinetic MFL defined by
\begin{align*}
dX_t &= V_t dt, \\
dV_t &= - V_t dt - D_m F\bigl(m_t^x, X_t\bigr) dt + \sqrt 2 dW_t,
&&\text{where $m^x_t = \Law(X_t)$,} \\
\intertext{and its associated \(N\)-particle system defined by}
dX^i_t &= V^i_t dt \\
dV^i_t &= - V^i_t dt - D_m F \bigl(\mu_{\mathbf X_t}, X^i_t\bigr) dt
+ \sqrt 2 dW^i_t,
&&\text{where $\mu_{\mathbf X_t} = \frac 1N \sum_{j=1}^N \delta_{X^j_t}$.}
\end{align*}
Here \(W_t\), \((W_t^i)_{i=1}^N\) are independent \(d\)-dimensional Brownian motions.
Denote \(m_t = \Law(X_t, V_t)\)
and \(m^N_t = \Law\bigl(X^1_t,\ldots,X^N_t,V^1_t,\ldots,V^N_t\bigr)\),
and we suppose the initial conditions
\(m_0\) and \(m^N_0\) have finite second moments.
We wish to show the convergence \(m^N_t \to m_t^{\otimes N}\)
when \(N \to +\infty\) in a uniform-in-\(t\) way.

We assume
\begin{itemize}
\item the mean field functional \(F\) is convex in the functional sense;
\item its intrinsic derivative \( (m,x) \mapsto D_m F(m,x)\)
is jointly Lipschitz with respect to the \(L^1\)-Wasserstein distance.
\item for every measure \(m \in \mathcal P_2(\mathbb R^d)\),
the probability measure proportional to
\(\exp\bigl(-\frac{\delta F}{\delta m}(m,x)\bigr) dx\)
satisfy a \emph{logarithmic Sobolev inequality} (LSI)
with a constant uniform in \(m\).
\item its second and third-order functional derivatives satisfy certain bounds.
\end{itemize}
Under these assumptions, we are able to obtain
\begin{itemize}
\item when \(t \to +\infty\), the mean field flow \(m_t\) converges
exponentially to the mean field invariant measure \(m_\infty\);
\item when \(t \to +\infty\), the \(N\)-particle flow \(m^N_t\) converges
approximately to the \(N\)-tensorized mean field invariant measure
\(m_\infty^{\otimes N}\), with an exponential rate uniform in \(N\);
\item if \(\frac 1N W_2^2\bigl(m^N_0, m_0^{\otimes N}\bigr) \to 0\)
when \(N \to +\infty\),
then \(\sup_{t \geq 0} \frac 1N W_2^2\bigl(m^N_t, m_t^{\otimes N}\bigr) \to 0\)
and \(\sup_{t \geq 1} \frac 1N H\bigl(m^N_t\big|m_t^{\otimes N}\bigr) \to 0\)
when \(N \to +\infty\).
\end{itemize}

\subsection{Related works}
\label{sec:related-works}

We give in this section a short review of the recent progresses
in the long-time behavior and the uniform-in-time propagation of chaos property
of McKean--Vlasov dynamics, with an emphasis on kinetic ones.
We refer readers to \cite{chaintron2022propagation1,chaintron2022propagation2}
for a more comprehensive review of propagation of chaos.

\paragraph{Coupling approaches.}

The coupling approach involves constructing a joint probability
of the mean field and \(N\)-particle systems to allow comparisons between them.
The \emph{synchronous coupling} method is employed in \cite{bolley2010trend}
and the uniform-in-time POC is shown by assuming the strong monotonicity
of the drift and the smallness of the mean field interaction.
The strong monotonicity is then relaxed by the \emph{reflection coupling} method
in \cite{egz2019}
and we refer readers to
\cite{schuh2022global,guillin2022convergence,kazeykina2020ergodicity}
for further developments.
Let us remark that the synchronous coupling gives often sharp contraction rates
under strong convexity assumptions,
while the reflection coupling allows us to treat dynamics of more general type
but gives far-from-sharp contraction rates.

\paragraph{Functional approaches.}

Another approach to uniform-in-time POC is the functional one,
and this is also the major approach of this paper.
In this situation in order to
study the long-time behaviors and propagation of chaos properties,
we construct appropriate (Lyapunov) functionals
and investigate the change of their values along the dynamics.
The relative entropy is used as the functional in \cite{monmarche2017long}
and its follow-up work \cite{guillin2021uniform}
to study kinetic McKean--Vlasov dynamics with regular interactions.
It is worth noting that the relative entropy approach has been successful
in handling singular interactions,
thanks to the groundbreaking work of Pierre-Emmanuel Jabin and Zhenfu Wang
\cite{jabin2018quantitative},
and we refer the readers to
\cite{guillin2021uniformvortex,decourcel2023sharp,rosenzweig2023global}
for recent developments.
However, we are not aware of any works using the relative entropy functional
(or its modifications) to study kinetic diffusions with singular interactions.

\paragraph{Comparison to \cite{ulpoc}.}

The present paper is a continuation of our previous work \cite{ulpoc},
where the overdamped version of mean field Langevin dynamics is studied,
and they share a number of key features.
We show the exponential convergence of the particle system using
the same componentwise decomposition of Fisher information
and the same componentwise log-Sobolev inequality.
The uniform-in-time propagation of chaos property
for both dynamics is then obtained by combining the exponential convergence
of the mean field and particle flow.
This paper is also different from \cite{ulpoc} in a number of aspects.
First, as the dynamics is generated by a hypoelliptic operator
instead of an elliptic one, we use hypocoercivity to recover
the exponential convergence.
Second, since we are not able to show hypercontractivity
of the kinetic dynamics (let alone reverse hypercontractivity),
we prove the entropic propagation of chaos
by studying its short-time regularization effects.
In this way we no longer restrict the initial condition of
the mean field dynamics,
but as a trade-off we require a higher-order regularity in measure of
the energy functional.
Finally, following in a remark in \cite{ulpoc},
we use an approximation argument to remove the condition
on the higher-order spatial derivatives in this work.

\subsection{Main contributions}
\label{sec:contributions}

\paragraph{Hypocoercivity for mean field systems.}

We extend the studies of the linear Fokker--Planck equation in \cite{hypocoer}
to the dynamics with general (but always regular) mean field interactions.
In particular, we do not suppose the interaction is
in form of a two-body potential,
which stands in contrast with \cite[Theorem 56]{hypocoer} and
\cite{monmarche2017long,guillin2021uniform}.
Moreover, in hypocoercive computations,
we find that the contributions from the mean field interaction
can always be dominated by the ``diagonal'' terms in the Fisher information,
already present in the case of linear dynamics.
Hence using the convexity of energy,
we are able to derive the hypocoercivity without restrictions
on the size of the interaction.
Furthermore, our assumptions imply a uniform-in-\(N\) bound
on the operator norm of the second-order derivatives of the effective potential
driving the \(N\)-particle system,
and the entropic hypocoercivity is consequently uniform in \(N\).
This is different from the situation of \(L^2\)-hypocoercivity,
where the condition given by Villani \cite[(7.3)]{hypocoer}
gives dimension-dependent constants
and therefore is unsuitable for studies of particle systems,
as remarked in \cite{guillin2021kinetic}.
Finally, let us mention that we derive the entropic hypocoercivity
under minimal regularity assumptions,
made possible by our approximation argument
(of functions \emph{and} of mean field functionals)
and the calculus in Wasserstein space developed in \cite{gf}.

\paragraph{Regularization in short time.}

We obtain two short-time regularization results
for the kinetic mean field dynamics.
The first, from Wasserstein to entropy, is a consequence of
the \emph{logarithmic Harnack's inequality},
obtained by applying the \emph{coupling by change of measure} method
of Panpan Ren and Feng-Yu Wang in \cite{ren2021exponential}
to the mean field and \(N\)-particle diffusions.
We remark also that very recently a similar inequality
(\cite[(3.13)]{huang2023coupling})
is proved for the propagation of chaos
of non-degenerate McKean--Vlasov diffusions.
The second regularization, from entropy to Fisher information,
is obtained by adapting Hérau's functional in \cite{hypocoer}
to our mean field setting and follows from the same hypocoercive computations
as we prove the convergence of the mean field flow.
We stress that although much stronger regularization phenomena are present,
for example from measure initial values to \(L^p\) for every \(p > 1\)
and to \(H^k\) for every \(k \geq 1\),
our results have the advantage of growing at most linearly in dimension,
making them suitable for studying the \(N\)-particle systems
under the limit \(N \to +\infty\).

\paragraph{Propagation of chaos.}

Finally, using the exponential convergence and the short-time regularizations,
we derive the propagation of chaos for the kinetic MFL,
i.e.\ bounds on the distances between the particle system
and the mean field system.
In particular, the initial value of the both systems can be arbitrary measures
of finite second moments without any further regularity constraints.
Moreover, the error terms do not have any dimension-dependence.
It is noteworthy that our approach allows us to not rely on a uniform-in-time
log-Sobolev inequality for the mean field flow,
and also that the dynamics considered are realized on the whole space instead of
the torus, standing in contrast with previous works,
e.g.\ \cite{guillin2021uniformvortex,lacker2023sharp,delarue2021uniform}.

\subsection{Notations}
\label{sec:notations}

Let \(d\) be a positive integer and \(x\), \(v\) be elements of \(\mathbb R^d\).
We denote the Euclidean norm of \(x\) and \(v\) by \(|x|\) and \(|v|\) respectively.
The letter \(z = (x,v)\) then denotes an element of \(\mathbb R^{2d}\)
with its Euclidean norm denoted by \(|z|^2 = |x|^2 + |v|^2\).
For a \(d \times d\) real matrix \(M\),
we denote by \(| M |_\textnormal{op}\) its operator norm
with respect to the Euclidean metric of \(\mathbb R^d\).
Let \(p \geq 1\).
Define \(\mathcal P_p(\mathbb R^d)\) to be the space of probabilities on \(\mathbb R^d\) of finite \(p\)-moment,
i.e.\ \(\mathcal P_p (\mathbb R^d) = \{ m \in \mathcal P (\mathbb R^d) : \int |x|^p m (dx) < +\infty\}\).
We denote the \(L^p\)-Wasserstein metric by \(W_p\) and refer readers to \cite[Chapter 7]{gf} for its definition and elementary properties.

Let \(F : \mathcal P_2(\mathbb R^d) \to \mathbb R\) be a mean field functional.
Denote by \(\frac{\delta F}{\delta m} : \mathcal P_2(\mathbb R^d) \times \mathbb R^d \to \mathbb R\) its linear functional derivative
and by \(D_m F = \nabla \frac{\delta F}{\delta m} : \mathcal P_2(\mathbb R^d) \times \mathbb R^d \to \mathbb R^d\) its intrinsic derivative, if they exist.
The definition of linear functional derivative on \(\mathcal P_2(\mathbb R^d)\) can be found in \cite[Definition 5.43]{probamfg1}.

Let \(X\) and \(Y\) be random variables.
We denote the distribution of \(X\) by \(\Law(X)\)
and say \(X \sim m\) if \(m = \Law(X)\).
We also say \(X \overset{d}= Y\) if \(\Law(X) = \Law(Y)\).

The set of couplings between probabilities \(\mu\) and \(\nu\)
is denoted by \(\Pi(\mu, \nu)\).

Let \(N \geq 2\) be integer.
The bold letters \(\mathbf x_N = (x^1,\ldots,x^N)\), \(\mathbf v_N = (v^1,\ldots,v^N)\) denote respectively \(N\)-tuples of elements in \(\mathbb R^{d}\)
and \(\mathbf z_N = (z^1,\ldots,z^N)\) an \(N\)-tuple of elements in \(\mathbb R^{2d}\).
We omit the subscript \(N\) when there are no ambiguities.
Given \(\mathbf x_N = (x^1,\ldots,x^N) \in \mathbb R^{dN}\), we denote the corresponding empirical measure by
\[
\mu_{\mathbf x_N} = \frac 1N \sum_{i=1}^N \delta_{x^i}.
\]
For \(i = 1\), \(\ldots\,\), \(N\), we define
\(-i = \{1,\ldots,N\} \setminus \{ i \}\), that is, the complementary index set,
and we denote the empirical measures formed by the \(N-1\) points
\((x_j)_{j \neq i}\) by
\[
\mu_{\mathbf x^{-i}_N} = \frac 1{N-1} \sum_{j \neq i} \delta_{x^j}.
\]
For an \(\mathbb R^{dN}\)-valued random variable \(\mathbf X_N = (X^i)_{i=1}^N\),
we can form the random empirical measures \(\mu_{\mathbf X_N}, \mu_{\mathbf X_N^{-i}}\).

Let \(I \subset \{1,\ldots,N\}\) and \(J = \{1,\ldots,N\} \setminus I\)
be the complementary index set.
Let \(\mathbf Z\) be an \(\mathbb R^{2dN}\)-valued random variable and
and \(m^N\) be its distribution, belonging to \(\mathcal P(\mathbb R^{2dN})\).
We denote the marginal and the (regular) conditional distributions of \(m^N\) by
\begin{align*}
m^{N,I} &= \Law (Z^i)_{i \in I}, \\
m^{N,I|J} (\mathbf z^J)
&= \Law \bigl( (Z^i)_{i \in I} \big| Z^j_t = z^j,~j \in J \bigr),
\end{align*}
where the latter is defined \(m^{N,J}\)-almost surely
and \(\mathbf z^J\) denotes the tuple \((z^j)_{j \in J}\).
We identify \(i\) with the singleton \(\{i\}\) when working with indices.

Whenever a measure \(m \in \mathcal P(\mathbb R^d)\) has a density with respect to the \(d\)-dimensional Lebesgue measure,
we denote its density function by \(m\) equally.
The relative \(H(\cdot | \cdot)\) between probabilities are always well defined
and the absolute entropy \(H(\cdot)\) is also well defined
if the measure in the argument has finite second moment.
If a measure \(m \in \mathcal P(\mathbb R^d)\)
has distributional derivatives \(Dm\) representable by a finite Borel measure
and \(Dm\) is absolutely continuous with respect to \(m\),
we define its Fisher information by
\[
I(m) = \int \biggl| \frac{Dm}{m} \biggr|^2 m,
\]
where \(\frac{Dm}{m}\) is the Radon--Nikodým derivative.
Otherwise we set \(I(m) = +\infty\).
One can verify that \(I(m)\) is finite only if \(m \in W^{1,1}(\mathbb R^d)\)%
\footnote{We sketch the proof here.
Suppose \(m\) has finite Fisher information.
Set \(m^n = m \star \rho^n\) for a mollifying sequence \((\rho^n)_{n \in \mathbb N}\).
Then we have \(\Vert m^n \Vert_{W^{1,1}} \leq C\) for all \(n \in \mathbb N\).
By Gagliardo--Nirenberg, \(m^n\) is uniformly bounded in \(L^p\) for some \(p > 1\),
so upon an extraction of subsequence, \((m^n)_{n \in \mathbb N}\) converges to some \(m' \in L^p\) weakly.
But \(m^n \to m\) in \(\mathcal P\).
The two limits coincide, i.e.\ \(m = m'\).
Hence \(m\) has density with respect to the Lebesgue
and so does \(Dm\).},
and in this case \(I(m) = \int \frac{|\nabla m|^2}{m}\),
\(\nabla m\) being the weak derivatives in \(L^1(\mathbb R^d; \mathbb R^d)\).
The Fisher information defined in this way corresponds
to the functional considered in \cite[(2.26)]{ambrosio2000functions}.
If \(m\) is a measure on \(\mathbb R^d\) having finite Fisher information,
and if \(\gamma\) is another measure on \(\mathbb R^d\)
having weakly differentiable density with respect to the Lebesgue,
we define the relative Fisher information by
\[
I(m | \gamma) = \int \biggl| \frac{\nabla m}{m}
- \frac{\nabla \gamma}{\gamma} \biggr|^2 m.
\]
For nonnegative functions \(f : \mathbb R^d \to [0,+\infty)\)
we define its entropy by
\[
\Ent_m f = \Expect_m [ f \log f ] - \Expect_m [ f ] \log \Expect_m [ f ],
\]
which is well defined in \([0,+\infty]\) by Jensen's inequality.

\paragraph{Organization of paper.}

In \cref{sec:assumptions-main-results}, we introduce our assumptions,
define the kinetic mean field Langevin and the particle system,
and state our main results.
We provide in \cref{sec:app} an exemplary dynamics modeling
neural networks' training and present our numerical experiments.
Moving on to the proofs,
we first show in \cref{sec:entropy-convergence} the exponential convergence
of the mean field and particle system dynamics.
We then study in \cref{sec:short-time-poc}
finite-time propagation of chaos and regularizations
of the kinetic MFL before combining all previous results
and showing the propagation of chaos theorem in its full form.
Finally, several technical results are proved in the appendices.

\section{Assumptions and main results}
\label{sec:assumptions-main-results}

\paragraph{Assumptions.}

Let \(F : \mathcal P_2(\mathbb R^d) \to \mathbb R\) be a mean field functional.
We suppose \(F\) is convex in the sense that
for every \(t \in [0,1]\) and every \(m\), \(m'\in \mathcal P_2(\mathbb R^d)\),
\begin{equation}
\label{eq:convex}
F\bigl((1-t)m + tm'\bigr) \leq (1-t)F(m) + tF(m').
\end{equation}
Suppose also its intrinsic derivative
\(D_m F : \mathcal P_2(\mathbb R^d) \times \mathbb R^d \to \mathbb R^d\)
exists and satisfies
\begin{multline}
\label{eq:lip}
\forall x,x' \in \mathbb R^d,~\forall m,m' \in \mathcal P_2(\mathbb R^d), \\
\lvert D_m F(m,x) - D_m F(m',x')\rvert
\leq M^F_{mm} W_1(m,m') + M^F_{mx} |x - x'|
\end{multline}
for some constants \(M^F_{mm}\), \(M^F_{mx} \geq 0\).
For each \(m \in \mathcal{P}_2(\mathbb R^d)\)
we define a probability measure \(\Pi^x(m)\) on \(\mathbb R^d\) by
\(\Pi^x (m)(dx) \propto \exp \bigl( - \frac{\delta F}{\delta m}(m,x)\bigr) dx\)
and suppose \(\Pi^x(m)\)
satisfies the \(\rho^x\)-\emph{logarithmic Sobolev inequality} (LSI),
uniformly in \(m\), for some \(\rho^x > 0\), that is,
for every \(m \in \mathcal P_2(\mathbb R^d)\),
\begin{equation}
\label{eq:x-lsi}
\forall f \in C^1_b(\mathbb R^d),\qquad
\rho^x \Ent_{\Pi^x(m)} (f^2) \leq \Expect_{\Pi^x(m)} \bigl[|\nabla f|^2\bigr].
\end{equation}
Finally for some of the results we suppose additionally
that \(F\) is third-order differentiable in measure
with \(\sup_{m \in \mathcal P_2(\mathbb R^d)}\sup_{x,x' \in \mathbb R^{d}}
\bigl|D_m^2 F(m,x,x')\bigr|_\textnormal{op} \leq M^F_{mm}\)
and
\begin{multline}
\label{eq:measure-third}
\forall m,m' \in \mathcal P_2(\mathbb R^{d}),~\forall x \in \mathbb R^{d}, \\
\biggl| \iint
\biggl[ \nabla_x \frac{\delta^3 F}{\delta m^3} (m',x, x', x') -
\nabla_x \frac{\delta^3 F}{\delta m^3} (m',x, x', x'') \biggr] m(dx') m(dx'')
\biggr| \\
\leq M^F_{mmm}
\end{multline}
for some constant \(M^F_{mmm}\).

\paragraph{Definition of \(\hat m\) and functional inequalities.}

For each \(m \in \mathcal P(\mathbb R^{2d})\),
we define \(\hat m\) to be the probability on \(\mathbb R^{2d}\) satisfying
\begin{equation}
\label{eq:hat-m}
\hat m(dxdv) \propto \exp \biggl( - \frac{\delta F}{\delta m}(m^x, x) - \frac 12 |v|^2 \biggr) dxdv,
\end{equation}
where \(m^x\) is the spatial marginal of \(m\).
Sometimes we will abuse the notation and
define for a measure \(m'^x \in \mathcal P_2(\mathbb R^d)\),
the probability
\(\widehat {m'^x}(dxdv) \propto \exp \bigl( - \frac{\delta F}{\delta m}(m'^x, x) - \frac 12 |v|^2 \bigr) dxdv \).
If \(F\) satisfies \cref{eq:x-lsi} with the LSI constant \(\rho^x\),
then setting
\begin{equation}
\label{eq:defn-rho}
\rho = \rho^x \wedge \frac 12,
\end{equation}
we have that the \(\rho\)-LSI holds for \(\hat m\):
for every \(f \in C^1_b (\mathbb R^{2d})\),
\begin{align}
\label{eq:lsi}
\rho \Ent_{\hat m} (f^2) &\leq \Expect_{\hat m} \bigl[|\nabla f|^2\bigr]. \\
\intertext{As a consequence, we have the \emph{Poincaré inequality}:
for every $f \in \mathcal C^1_b(\mathbb R^{2d})$,}
\label{eq:poincare}
2\rho \Var_{\hat m} (f) &\leq \Expect_{\hat m} \bigl[|\nabla f|^2\bigr]; \\
\intertext{and \emph{Talagrand's \(T_2\) transport inequality}:
for every $\mu \in \mathcal P_2(\mathbb R^{2d})$,}
\label{eq:t2}
2\rho W_2^2(\mu, \hat m) &\leq H(\mu | \hat m).
\end{align}

\paragraph{Mean field and particle system.}

We study the \emph{mean field kinetic Langevin dynamics}, that is,
the following McKean--Vlasov SDE
\begin{equation}
\begin{aligned}
dX_t &= V_t dt, \\
dV_t &= - V_t dt - D_m F\bigl(m_t^x, X_t\bigr) dt
+ \sqrt 2 dW_t, \quad\text{where $m^x_t = \Law(X_t)$.}
\end{aligned}
\label{eq:mf-sde}
\end{equation}
Let \(N \geq 2\).
The corresponding \(N\)-\emph{particle system} is defined by
\begin{equation}
\label{eq:ps-sde}
\begin{aligned}
dX^i_t &= V^i_t dt \\
dV^i_t &= - V^i_t dt - D_m F \bigl(\mu_{\mathbf X_t}, X^i_t\bigr) dt
+ \sqrt 2 dW^i_t,
\quad\text{where $\mu_{\mathbf X_t} = \frac 1N \sum_{j=1}^N \delta_{X^j_t}$.}
\end{aligned}
\end{equation}
Here \(W\) and \(W^i\) are standard Brownian motions in \(\mathbb R^d\),
and \((W^i)_{i=1}^N\) are independent.
Their marginal distributions \(m_t = \Law(X_t)\), \(m^N_t = \Law(\mathbf X_t) = \Law(X^1_t,\ldots,X^N_t)\) solve respectively the Fokker--Planck equations:
\begin{align}
\label{eq:mf-fp}
\partial_t m &= \Delta_{v} m + \nabla_{v} \cdot (m v) - v \cdot \nabla_{x} m
+ D_m F \bigl( m^x_t , x \bigr) \cdot \nabla_{v} m , \\
\label{eq:ps-fp}
\partial_t m^N &= \sum_{i=1}^N
\Bigl( \Delta_{v^i} m^N + \nabla_{v^i} \cdot (m^N v^i) - v^i \cdot \nabla_{x^i} m^N  + D_m F ( \mu_{\mathbf x} , x^i ) \cdot \nabla_{v^i} m^N \Bigr),
\end{align}
where on the second line \(\mu_{\mathbf x} \coloneqq \frac 1N \sum_{i=1}^N \delta_{x^i}\).
The mean field equation \cref{eq:mf-fp} is non-linear while the \(N\)-particle system equation \cref{eq:ps-fp} is linear.
We will show in \cref{lem:wellposedness-regularity} the wellposedness
of the mean field dynamics \cref{eq:mf-fp}
with initial conditions of finite second moment.

\begin{rem}
\label{rem:vol-scaling}
We have fixed the volatility and the friction constants to simplify the computations.
In order to apply our results to the diffusion process defined by
\begin{equation}
\label{eq:mf-sde-noscaling}
\begin{aligned}
dX_t &= \alpha V_t dt, \\
dV_t &= - \gamma V_t dt - D_m F\bigl(\Law(X_t), X_t\bigr) dt + \sigma dW_t,
\end{aligned}
\end{equation}
with \(\alpha\), \(\gamma\), \(\sigma > 0\),
we introduce the new variables:
\[
x' = \frac{(2\gamma^3)^{1/2}}{\alpha \sigma} x,\quad
v' = \frac{(2\gamma)^{1/2}}{\sigma} v,\quad
t' = \gamma^{-1} t,
\]
define \(m'\) to be the push-out of measure \(m\) under \(x \mapsto x'\),
and set
\[
F'(m') = \biggl(\frac{2}{\gamma\sigma^2}\biggr)^{\!1/2} F(m).
\]
Then the stochastic process \(t' \mapsto (X'_{t'}, V'_{t'})\) satisfy
\[
\begin{aligned}
dX'_{t'} &= V'_{t'} dt', \\
dV'_{t'} &= - V'_{t'} dt' - D_m F' \bigl(\Law(X'_{t'}), X'_{t'} \bigr) dt + \sqrt 2 dW'_{t'},
\end{aligned}
\]
where \(W'_{t'} \coloneqq \gamma^{1/2} W_{t}\) is a standard Brownian motion.
In the same way we can treat the particle system defined by
\begin{equation}
\label{eq:ps-sde-noscaling}
\begin{aligned}
dX^i_t &= \alpha V^i_t dt \\
dV^i_t &= - \gamma V^i_t dt - D_m F \bigl(\mu_{\mathbf X_t}, X^i_t\bigr) dt
+ \sigma dW^i_t,
\quad\text{where $\mu_{\mathbf X_t} = \frac 1N \sum_{j=1}^N \delta_{X^j_t}$.}
\end{aligned}
\end{equation}
\end{rem}

\paragraph{Free energies and invariant measures.}

For measures \(m \in \mathcal P_2(\mathbb R^{2d})\),
\(m^N \in \mathcal P_2(\mathbb R^{2dN})\),
we introduce the mean field and \(N\)-particle free energies:
\begin{align}
\label{eq:mf-free-energy}
\mathcal F(m) &= F(m^x) + \frac 12 \int |v|^2 m(dx dv) + H(m), \\
\label{eq:ps-free-energy}
\mathcal F^N(m^N) &= \int \biggl( NF (\mu_{\mathbf x}) + \frac 12 |\mathbf v|^2 \biggr) m^N(d\mathbf xd\mathbf v) + H(m^N).
\end{align}
The functionals are well defined with values in \(( -\infty, +\infty ]\).
We will also work with probability measures,
\(m_\infty \in \mathcal P_2(\mathbb R^{2d})\)
and \(m^N_\infty \in \mathcal P_2(\mathbb R^{2dN})\),
satisfying
\begin{align}
\label{eq:mf-invariant-measure}
m_\infty(dxdv) &\propto \exp \biggl( -\frac{\delta F}{\delta m}(m^x_\infty, x)
- \frac 12 |v|^2 \biggr) dxdv, \\
\label{eq:ps-invariant-measure}
m^N_\infty(d\mathbf x d\mathbf v) &\propto \exp \biggl( - N F (\mu_{\mathbf x})
- \frac 12 |\mathbf v|^2 \biggr) d\mathbf xd\mathbf v,
\end{align}
and having finite exponential moments, i.e., the integrals
\(\int \exp \bigl(\alpha (|x| + |v|) \bigr) m_\infty(dx dv)\)
and
\(\int \exp \bigl(\alpha (|\mathbf x| + |\mathbf v|) \bigr)
m^N_\infty(d\mathbf x d\mathbf v)\)
are finite for every \(\alpha \geq 0\).
We call \(m_\infty\), \(m^N_\infty\) \emph{invariant} measures to
the dynamics \cref{eq:mf-fp,eq:ps-fp} respectively.
The existence and uniqueness of the invariant measures
are guaranteed by our assumptions \cref{eq:convex,eq:lip,eq:x-lsi},
as will be stated in \cref{lem:exist-unique-invariant-measures}.

\paragraph{Main results.}

Recall that \(m_t\) and \(m^N_t\) are the respective marginal distributions of the mean field and the \(N\)-particle system \cref{eq:mf-sde,eq:ps-sde}.
We first prove the exponential entropic convergence result for the MFL dynamics \cref{eq:mf-sde}.

\begin{thm}[Entropic convergence of MFL]
\label{thm:mf-entropy-convergence}
Assume \(F\) satisfies \cref{eq:convex,eq:lip,eq:x-lsi}.
If \(m_{0}\) has finite second moment, finite entropy and finite Fisher information,
then there exist constants
\[
C_0 = C_0\bigl(M^F_{mx},M^F_{mm}\bigr),\qquad
\kappa = \kappa\bigl(\rho^x,M^F_{mx},M^F_{mm}\bigr)
\]
such that for every \(t \geq 0\),
\begin{equation}
\label{eq:mf-entropy-convergence}
\mathcal F(m_t) - \mathcal F(m_\infty) \leq \bigl(\mathcal F(m_0) - \mathcal F(m_\infty) + C_0 I(m_0 | \hat m_0)\bigr) e^{-\kappa t}.
\end{equation}
\end{thm}

The proof of the theorem is postponed to \cref{sec:mf}.
We note that the proof only relies on the $W_2$-Lipschitz continuity
of $m \mapsto D_m F(m,x)$, contrary to the $W_1$ one stated in \cref{eq:lip}.

Our second major contribution is the uniform-in-\(N\) exponential entropic convergence of the particle systems.

\begin{thm}[Entropic convergence of particle systems]
\label{thm:ps-entropy-convergence}
Assume \(F\) satisfies \cref{eq:convex,eq:lip,eq:x-lsi}.
If \(m^N_{0}\) has finite second moment, finite entropy and finite Fisher information
for some \(N \geq 2\),
then there exist constants
\[
C_0 = C_0\bigl(M^F_{mx},M^F_{mm}\bigr),\quad
C_1 = C_1\bigl(\rho^x,M^F_{mx},M^F_{mm}\bigr),\quad
\kappa = \kappa\bigl(\rho^x,M^F_{mx},M^F_{mm}\bigr)
\]
such that if \(N > C_1/\kappa\),
then for every \(t \geq 0\),
\begin{multline}
\mathcal F^N(m^N_t) - N\mathcal F(m_\infty)
\leq
\Bigl(\mathcal F\bigl(m^N_0\bigr)
- N\mathcal F(m_\infty) + C_0I\bigl(m^N_0\big|m^N_\infty\bigr) \Bigr)
e^{ - ( \kappa - C_1/N ) t } \\
+ \frac{C_1 d}{\kappa - C_1/N}.
\label{eq:ps-entropy-convergence}
\end{multline}
\end{thm}

The proof of the theorem is postponed to \cref{sec:ps}.

\begin{rem}
Strictly speaking, the result \cref{eq:ps-entropy-convergence} does not imply that the particle systems converge uniformly.
We only show \(\frac1N\mathcal F^N\bigl(m^N_t\bigr)\)
approaches the mean field minimum \(\mathcal F(m_\infty)\) uniformly quickly
until they are \(O(N^{-1})\)-close to each other.
\end{rem}

\begin{rem}
\Cref{thm:mf-entropy-convergence,thm:ps-entropy-convergence}
state results concerning the convergence of the respective free energies,
which we will also call ``convergence of entropy'' or ``entropic convergence'',
since in both cases the differences of free energies are related to relative entropies,
as shown in \cref{lem:mf-entropy-sandwich,lem:ps-entropy-sandwich}.
\end{rem}

We now present the main theorem,
which establishes the uniform-in-time propagation of chaos in both the Wasserstein distance and the relative entropy.
The results are direct consequences of the exponential convergence in
\cref{thm:mf-entropy-convergence,thm:ps-entropy-convergence}
and the regularization phenomena to be studied in \cref{sec:short-time-poc}.

\begin{thm}[Wasserstein and entropic propagation of chaos]
\label{thm:poc}
Assume \(F\) satisfies \cref{eq:convex,eq:lip,eq:x-lsi,eq:measure-third}.
If \(m_0\) belongs to \(\mathcal P_2(\mathbb R^d)\)
and \(m^N_0\) belongs to \(\mathcal P_2(\mathbb R^{dN})\)
for some \(N \geq 2\),
then there exist constants
\(C_1 = C_1\bigl(M^F_{mx}, M^F_{mm}, M^F_{mmm}\bigr)\),
\(C_2 = C_2\bigl(\rho^x,M^F_{mx},M^F_{mm}\bigr)\)
and \(\kappa = \kappa\bigl(\rho^x,M^F_{mx},M^F_{mm}\bigr)\)
such that if \(N > C_2/\kappa\),
then for every \(t > 0\),
\begin{align}
\MoveEqLeft W_2^2\bigl(m^N_t, m^{\otimes N}_t\bigr) \nonumber \\
&\leq \min \biggl\{
C_1W_2^2 \bigl(m^N_0, m_0^{\otimes N}\bigr) e^{C_1 t}
+ C_1 (e^{C_1 t} - 1) (\Var m_0 + d), \nonumber \\
&\phantom{\leq \min \biggl\{}
\frac{C_2N}{(t \wedge 1)^6} W_2^2(m_0, m_\infty) e^{-\kappa t}
+ \frac{C_2}{(t \wedge 1)^6} W_2^2\bigl(m^N_0, m^{\otimes N}_\infty\bigr)
e^{ - ( \kappa - C_2/N ) t } \nonumber \\
&\phantom{\leq \min \biggl\{}
\hspace{17.5em}+ \frac{C_2d}{\kappa - C_2/N} \biggr\};
\label{eq:w-poc}
\end{align}
moreover, for every \(t\) and \(s\) such that \(s + 1 \geq t > s \geq 0\),
\begin{multline}
\label{eq:h-poc}
H\bigl(m^N_t \big| m^{\otimes N}_t\bigr)
\leq
\frac{C_1}{(t - s)^3} W_2^2(m^N_s, m^{\otimes N}_s) + C_1 (e^{C_1 (t - s)} - 1)
(\Var m_s + d).
\end{multline}
\end{thm}

The proof of the theorem is postponed to \cref{sec:poc}.

\paragraph{Comments on the assumptions.}

Compared to our previous work \cite{ulpoc},
we have removed the technical assumption
that \(x \mapsto D_m F(m,x)\) has bounded higher-order derivatives
by a mollifying procedure of the mean field functional.
However, the spatial Lipschitz constant \(M^F_{mx}\),
appearing in the assumption \cref{eq:lip},
will contribute to the constants, especially the rate of convergence \(\kappa\),
in our theorems.
Nevertheless, this behavior is expected for kinetic dynamics,
as this dependency is already present for the linear Fokker--Planck dynamics in \cite{hypocoer}.
Finally, we introduce the new condition \cref{eq:measure-third}
on the second and third-order derivatives in measure of the mean field functional.
The condition \cref{eq:measure-third} is used to obtain \(O(1)\) errors
in the propagation of chaos bounds \cref{eq:w-poc,eq:h-poc} in \cref{thm:poc},
which are stronger than the dimension-dependent errors
obtained from the method of Fournier and Guillin \cite{fournier2015rate}.

\section{Application: training neural networks by momentum gradient descent}
\label{sec:app}

We have given in Section 3 of our previous work \cite{ulpoc}
several examples of mean field functionals
satisfying conditions \cref{eq:convex,eq:lip,eq:x-lsi} of our theorems,
and the only additional condition that remains to verify
is the bound on the higher-order measure derivative \cref{eq:measure-third}.
In the following we will recall the mean field formulation of
the loss of two-layer neural networks
and its corresponding kinetic dynamics (see \cite[Examples 2 and 4]{ulpoc}),
and verify that it satisfies indeed the additional assumption.

\subsection{Mean field formulation of neural network}
\label{ex:train_nns}

Recall that the structure of a two-layer neural network is determined
by its \emph{feature map}:
\[
\mathbb R^d \ni z \mapsto \Phi(\theta; z)
\coloneqq \ell(c) \varphi (a \cdot z + b) \in \mathbb R^{d'},
\]
where \(\theta \coloneqq (c,a,b) \in \mathbb R^{d'} \times \mathbb{R}^{d} \times \mathbb R \eqqcolon S\)
is the parameter of a single neuron,
\(\varphi : \mathbb R \to \mathbb R\) is a non-linear \emph{activation function}
satisfying the squashing condition (see \cite[(3.4)]{ulpoc}),
and \(\ell : \mathbb{R} \to [-L,L]\) is a \emph{truncation function}
with \emph{threshold} \(L \in (0,+\infty)\).
Here the action of the truncation is tensorized:
\(\ell (c) = \ell (c^1, \ldots, c^{d'})
\coloneqq \bigl(\ell(c^1), \ldots, \ell(c^{d'})\bigr)\)
for a \(d'\)-dimensional vector \(c = (c^1, \ldots, c^{d'})\).
Then given \(N\) neurons with respective parameters
\(\theta^1\), \(\ldots\,\), \(\theta^N\), the associated network's output reads
\begin{equation}
\label{eq:nn-ps-feature}
\mathbb R^d \ni z \mapsto \Phi^N (\theta^1,\ldots,\theta^N ; z) = \frac 1N \sum_{i=1}^N \Phi(\theta^i; z) \in \mathbb R.
\end{equation}
Here \(z\) should be considered as the input of the network,
i.e.\ the \emph{feature}, and the value \(\Phi^N(\theta^1, \ldots, \theta^N; z)\)
should correspond to the \emph{label}.
We wish to find the optimal neuron parameters \((\theta_i)_{i=1}^N\)
for a possibly unknown distribution of feature-label tuples
\(\mu \in \mathcal P (\mathbb R^{d + d'})\).
In order to quantify the goodness of networks, we define the \emph{loss}:
\begin{equation}
\label{eq:nn-ps-loss}
F_\textnormal{NNet}^N (\theta^1,\ldots,\theta^N) =
\frac N2 \int \lvert y - \Phi^N(\theta^1,\ldots,\theta^N; z)\rvert^2 \mu(dzdy).
\end{equation}
It is proposed in \cite{HRSS19, ulpoc} that instead of minimizing
the original loss \cref{eq:nn-ps-loss},
we consider the mean field output function
\(\Expect^{\Theta \sim m} [\Phi (\Theta; \cdot)]\)
and minimize the mean field loss
\begin{equation}
\label{eq:nn-mf-loss}
F^N_\textnormal{NNet} (m) = \int \bigl| y - \Expect^{\Theta \sim m} [\Phi (\Theta; z)] \bigr|^2 \mu (dz dy).
\end{equation}
We also add a quadratic regularizer
\[
F_\textnormal{Ext}(m) = \frac{\lambda}{2} \int |\theta|^2 m(d\theta)
\]
with regularization parameter \(\lambda >0\).
The final optimization problem then reads
\begin{equation}
\label{eq:nn-mf-optimize}
\inf_{m\in\mathcal{P}_2(S)} F(m) \coloneqq F_\textnormal{NNet}(m) + F_\textnormal{Ext}(m).
\end{equation}
Following the calculations in \cite{ulpoc} we can show that
if both the truncation and activation function are bounded
and has bounded derivatives of up-to-second order,
then the conditions \cref{eq:convex,eq:lip,eq:x-lsi} are verified.
Finally, the third-order derivatives \(\frac{\delta^3 F}{\delta m^3}\)
is a constant thanks to the fact that the loss function is quadratic,
and therefore the condition \cref{eq:measure-third}
is satisfied with \(M^F_{mmm} = 0\).

\begin{rem}
Following \cite[Remark 3.6]{ulpoc}, we recognize that the SDE \cref{eq:mf-sde}
describes the continuous version of the gradient descent algorithm with momentum.
Among various momentum gradient descent methods
commonly used to train neural networks,
the most prevalent ones are RMSProp and Adam algorithms
(see \cite{hinton2012neural,kingma2014adam}),
where the momentum is accumulated and the step size is adapted along the dynamics.
In \cite{liu2020improved,sebbouh2021almost,reddi2019convergence}
the authors studied the convergence of these momentum-based algorithms
and compared them to algorithms without momentum based on optimization theory.
We note that estimates of the discretization error
and optimal parameters can also be found in these studies.
\end{rem}

\subsection{Numerical experiments}
\label{sec:numerical}

We present our numerical experiments in this section.
Our experiments are based on the discretized version of
a particle system dynamics \cref{eq:ps-sde-noscaling}.
We first explain the optimization problem and the numerical algorithm,
and then present our two experiments:
the first investigates the convergence behavior as the number of particles
tends to infinity,
and the second compares the kinetic dynamics
to the corresponding overdamped dynamics.

\paragraph{Problem setup and momentum algorithm.}
We aim to solve a supervised learning problem:
our goal is to classify the handwritten digits ``\(4\)'' and ``\(6\)''
by a two-layer neural network.
We randomly choose
\(K = 10^4\) samples from the MNIST dataset \cite{mnist}
and denote by \((z_k)_{k=1}^K\) the figures in \(28 \times 28\) pixel format,
i.e.\ each \(z_k\) belongs to \(\mathbb R^{28 \times 28} = \mathbb R^{784}\),
and by \((y_k)_{k=1}^K\) the \emph{one-hot} vectors for the two classes of digits,
i.e.\ if the \(k\)-th figure corresponds to the digit ``\(4\)'',
then \(y_k = (1, 0)\), otherwise \(y_k = (0, 1)\).
See \cref{fig:mnist-sample} for random samples in the dataset.
We choose \(N\) particles
and use the sigmoid function as the activation,
i.e.\ \(\varphi(x) = 1\big/\bigl(1+\exp(-x)\bigr)\).
The truncation function is fixed by
\[
\ell(x) = L \tanh (x/L) = L\, \frac{\exp(2x/L)-1}{\exp(2x/L) + 1}
\]
and its threshold equals \(L\).
The quadratic regularization parameter is denoted by \(\lambda\).
Following the arguments of \cite{ulpoc} and the precedent section,
all the conditions of our theorems
\cref{eq:convex,eq:lip,eq:x-lsi,eq:measure-third} are satisfied.
In the beginning of training process, the neuron positions
\((\Theta^i_0)_{i=1}^N = (C^{x,i}_0, A^{x,i}_0, B^{x,i}_0)_{i=1}^N\)
and momenta \((\Psi^i_0)_{i=1}^N = (C^{v,i}_0, A^{v,i}_0, B^{v,i}_0)_{i=1}^N\)
are sampled independently from a given initial distribution
\(m_0^x\),
\(m_0^v \in \mathcal P(\mathbb R^2\times \mathbb R^{784} \times \mathbb R)\).
We update the parameters \((\Theta^i_0)_{i=1}^N\)
and \((\Psi^i_0)_{i=1}^N\) following the discrete-time version
of the underdamped Langevin SDE \cref{eq:mf-sde-noscaling}
with fixed set of parameters \((\alpha, \gamma, \sigma)\),
that is, we calculate the neurons' evolution by \cref{alg:nn-nmgd}.

\begin{figure}
\centering
\includegraphics[width=0.8\linewidth]{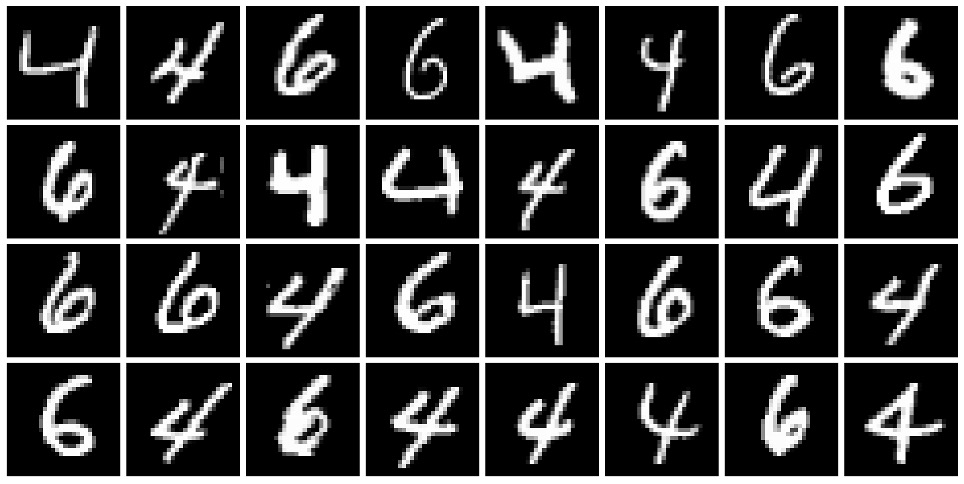}
\caption{Randomly chosen handwritten digits ``\(4\)'' and ``\(6\)'' from the MNIST dataset.}
\label{fig:mnist-sample}
\end{figure}

\paragraph{Convergence when \(N \to +\infty\).}

To study the behavior of the momentum training dynamics when \(N \to +\infty\)
we conduct independent experiments with the an increasing number of particles:
\(N = 2^P\) for \(P=5\), \(6\), \(\ldots\,\), \(10\)
and repeat the experiment $10$ times for each $N$.
The hyperparameters for this experiment
are listed in the second column of \cref{tab:hypara}.

To quantity the convergence,
we compute \(\frac 1N F_\textnormal{NNet}^N \bigl(\Theta^1_t,\ldots,\Theta^N_t\bigr)\)
and \(\frac{1}{N}F_\textnormal{Kinet}^N\bigl(\Psi^1_t,\ldots,\Psi^N_t\bigr)\),
where \(F_\textnormal{Kinet}^N(\Psi^1,\ldots,\Psi^N)
\coloneqq \frac 12 \sum_{i=1}^N |\Psi^i|^2\).
We then compute its average of the respective quantities
over the \(10\) repeated runs.
The evolutions of \(\frac 1NF_\textnormal{NNet}^N\) and
\(\frac 1NF_\textnormal{NNet}^N + \frac 1NF_\textnormal{Kinet}^N\)
are plotted in \cref{fig:loss_withoutl2}
and \cref{fig:loss_withl2} respectively,
and can be characterized by two distinct phases.
In the first phase, both the quantities decrease
and the second quantity decreases exponentially, for every \(N\).
We also find that in this phase
the convergence rates are almost the same for different \(N\)
and this is coherent with the behavior indicated by our theoretical upper bound
\cref{eq:ps-entropy-convergence}.
We also observe that \(\frac 1N F_\textnormal{NNet}^N\) fluctuates in a stronger way
than \(\frac 1N F_\textnormal{NNet}^N + \frac 1N F_\textnormal{Kinet}^N\).
In the second phase, both the values cease to decrease
but the remnant values differ for different \(N\).

To investigate the relationship between the remnant values in the second phase
and the number of particles \(N\),
we compute the average value of
\(\frac 1N F_\textnormal{NNet}^N  + \frac{1}{N}F_\textnormal{Kinet}^N\)
of the last \(500\) training epochs for each individual run
and plot their values in \cref{fig:trainedloss}.
Motivated by the upper bound \cref{eq:ps-entropy-convergence} in \cref{thm:ps-entropy-convergence},
we fit the remnant values by \(C' + \frac{C}{N}\)
and find the values are well fitted by this curve.

\paragraph{Comparison to algorithm without momentum.}

We also investigate the difference between gradient descent algorithms
with and without momentum by working on the same set of hyperparameters,
listed in the last column of \cref{tab:hypara}.
It is found that the algorithm with momentum leads to much stronger fluctuations
compared the algorithm without momentum (see \cref{fig:tosgd}).
Both algorithms cease to decrease after certain training epochs,
but the momentum algorithm leads to better loss in the end.
This may be explained by the fact that the presence of momentum helps
the particles to escape local minima.

\begin{table}[htbp]
\centering
\begin{tabular}{ccc}
\toprule
Hyperparameter & First Exp.'s Value   & Second Exp.'s Value  \\
\midrule
$N$            & $[128,256,512,1024,2048]$ & $256$           \\
$\Delta t$     & $0.02$               & $0.01$               \\
$T$	           & $300$                & $500$                \\
$m_0^x$        & $\mathcal N(0,0.01)$ & $\mathcal N(0,0.01)$ \\
$m_0^v$        & $\mathcal N(0,0.25)$ & $\mathcal N(0,0.01)$ \\
$L$            & $500$                & $500$                \\
$\lambda$      & $10^{-4}$            & $10^{-3}$                  \\
$\alpha$       & $1$                  & $1$                  \\
$\gamma$       & $0.1$                & $0.1$                \\
$\sigma$       & $0.01\sqrt 2$        & $0.01\sqrt 2$        \\
\bottomrule
\end{tabular}
\caption{Hyperparameters of neural networks' trainings.}
\label{tab:hypara}
\end{table}

\begin{algorithm}
\caption{Noised momentum gradient descent for training a two-layer neural network}
\label{alg:nn-nmgd}
\KwIn{number of particles $N$,
truncation \(L\),
data set $(z_k,y_k)_{k=1}^K$,
noise $\sigma$,
friction $\gamma$,
$l_2$ regularization $\lambda$,
initial distribution $(m_0^x,m_0^v)$,
time step $\Delta t$,
time horizon $T$
}
\KwOut{$(\Theta^i_{T})_{i=1}^N$}
generate i.i.d.\
$\Theta_0^{i} = \bigl(A^{x,i}_0,B^{x,i}_0,C^{x,i}_0\bigr) \sim m_0^x$
for $i = 1$, \dots, $N$\;
generate i.i.d.\ $\Psi_0^{i} = \bigl(A^{v,i}_0,B^{v,i}_0,C^{v,i}_0\bigr) \sim m_0^v$
for $i = 1$, \dots, $N$\;
\For{$t=0$, $\Delta t$, $2\Delta t$, \dots, $T-\Delta t$}{
generate i.i.d.\ $\mathcal{N}^i_{t} \sim \mathcal N(0,1)$
for $i=1$, \dots, $N$\;
\tcp{update particles according to discretized underdamped Langevin}
\For{$i = 1$, \dots, $N$}{
$\begin{multlined}[b]\textstyle\Psi^i_{t+\Delta t} \leftarrow (1-\gamma \Delta t)\Psi^i_{t}
- D_m F_\textnormal{NNet} \bigl(\frac{1}{N}\sum_{j=1}^N \delta_{\Theta^j_{t}},\Theta^i_{t} \bigr) \Delta t \\
- \lambda \Theta^i_{t} \Delta t
+ \sigma \sqrt{\Delta t} \mathcal{N}^i_{t}\end{multlined}$\;
$\Theta^i_{t+\Delta t} \leftarrow \Theta^i_{t} + \Psi^i_{t+\Delta t} \Delta t  $\;
\tcp{where
\(\begin{multlined}[t]\textstyle D_m F_\textnormal{NNet} \bigl( \frac 1N \sum_{j=1}^N \delta_{\Theta^j_t}, \Theta^i_t \bigr) \\
\textstyle = \frac{1}{K} \sum_{k=1}^K \bigl(y_k - \Psi^N( \Theta^1_t, \ldots, \Theta^N_t; z_k) \bigr)
\frac{\partial \Psi}{\partial \theta} (\Theta^i_t; z_k)\end{multlined}\)}}}
\end{algorithm}

\begin{figure}[htbp]
\begin{minipage}{0.45\linewidth}
\centering
\includegraphics[width=\linewidth]{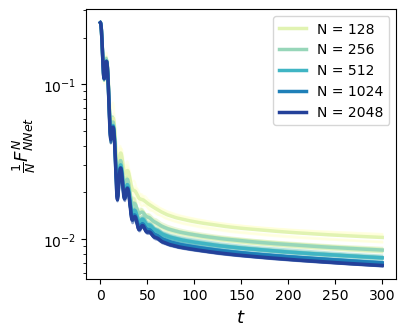}
\captionof{figure}{Individual (shadowed) and \(10\)-averaged (bold) losses without kinetic energy versus time.}
\label{fig:loss_withoutl2}
\end{minipage}
\hfill
\begin{minipage}{0.45\linewidth}
\centering
\includegraphics[width=\linewidth]{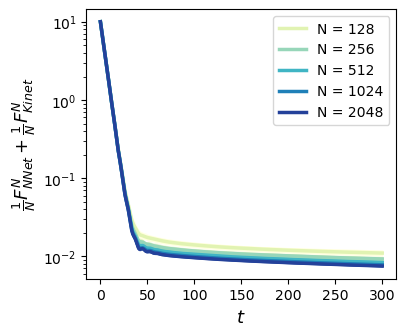}
\captionof{figure}{Individual (shadowed) and \(10\)-averaged (bold) losses with kinetic energy versus time.}
\label{fig:loss_withl2}
\end{minipage}
\end{figure}

\begin{figure}[htbp]
\begin{minipage}{0.48\linewidth}
\centering
\includegraphics[width=\linewidth]{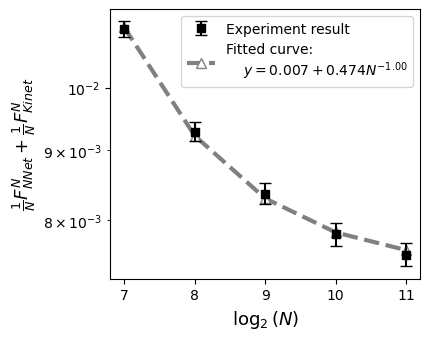}
\captionof{figure}{Average values of
$\frac{1}{N}F_{\textnormal{NNet}} + \frac{1}{N}F_{\textnormal{Kinet}}$
over the last \(500\) epochs.
The mean (black squares) and standard derivations (error bars) are calculated from the \(10\) independent runs.
Dashed curve fits the data.
}
\label{fig:trainedloss}
\end{minipage}
\hfill
\hfill
\begin{minipage}{0.45\linewidth}
\centering
\includegraphics[width=\linewidth]{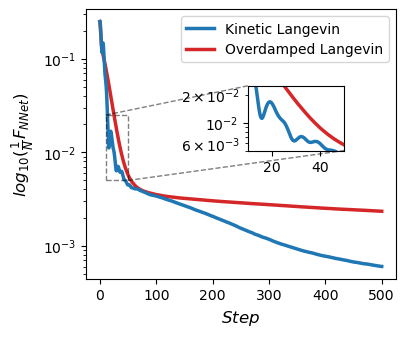}
\captionof{figure}{Target function $\frac{1}{N}F_{\textnormal{NNet}}$ for underdamped Langevin (blue) and overdamped Langevin (red) versus time.}
\label{fig:tosgd}
\end{minipage}

\end{figure}

\section{Entropic convergence}
\label{sec:entropy-convergence}

\subsection{Collection of known results}
\label{sec:collection-results}

Before moving on to the proofs,
we first state some elementary results without proofs.
They are either immediate consequences of the corresponding ones in our previous work \cite{ulpoc},
or easy adaptations thereof.

\begin{lem}[Existence and uniqueness of invariant measures]
\label{lem:exist-unique-invariant-measures}
If \(F\) satisfies \cref{eq:convex,eq:lip,eq:x-lsi},
then there exist unique measures \(m_\infty\) and \(m^N_\infty\)
satisfying \cref{eq:mf-invariant-measure,eq:ps-invariant-measure} respectively
and they have finite exponential moments.
\end{lem}

\begin{lem}[Mean field entropy sandwich]
\label{lem:mf-entropy-sandwich}
Assume \(F\) satisfies \cref{eq:convex,eq:lip,eq:x-lsi}.
Then for every \(m \in \mathcal P_2(\mathbb R^{2d})\),
we have
\begin{multline}
\label{eq:mf-entropy-sandwich}
H(m | m_\infty) \leq \mathcal F(m) - \mathcal F(m_\infty)
\leq H(m | \hat m) \\
\leq \biggl(1 + \frac{M^F_{mm}}{\rho} + \frac{(M^F_{mm})^2}{2\rho^2} \biggr)
H(m | m_\infty),
\end{multline}
where \(\rho\) is defined by \cref{eq:defn-rho}.
Here, the leftmost inequality holds even without the uniform LSI condition
\cref{eq:x-lsi}, once there exists a measure $m_\infty$
satisfying \cref{eq:mf-invariant-measure} and having finite exponential moments.
\end{lem}

\begin{lem}[Particle system's entropy inequality]
\label{lem:ps-entropy-sandwich}
Assume that \(F\) satisfies \cref{eq:convex}
and that there exists a measure \(m_\infty \in \mathcal P_2(\mathbb R^{2d})\)
verifying \cref{eq:mf-invariant-measure}.
Then for all \(m^N \in \mathcal P_2(\mathbb R^{dN})\) of finite entropy, we have
\begin{equation}
\label{eq:ps-entropy-sandwich}
H(m^N | m_\infty^{\otimes N})
\leq \mathcal F^N(m^N) - N \mathcal F(m_\infty).
\end{equation}
\end{lem}

\begin{lem}[Information inequalities]
\label{lem:info-ineqs}
Let \(X_1\), \(\ldots\,\), \(X_N\) be measurable spaces,
\(\mu\) be a probability on the product space
\(X = X_1 \times \cdots \times X_N\)
and \(\nu = \nu^1 \otimes \cdots \otimes \nu^N\) be a \(\sigma\)-finite measure.
Then
\begin{equation}
\label{eq:info-ineqs}
\sum_{i=1}^N H(\mu^i | \nu^i) \leq H(\mu | \nu) \leq \sum_{i=1}^N
\int H\bigl(\mu^{i|-i}(\cdot|\mathbf x^{-i}) \big| \nu^i\bigr)
\mu^{-i} (d\mathbf x^{-i}).
\end{equation}
Here we set the rightmost term to \(+\infty\) if the conditional distribution \(\mu^{i|-i}\) does not exist \(\mu^{-i}\)-a.e.
\end{lem}

\subsection{Mean field system}
\label{sec:mf}

In this section we study the mean field system described by
the Fokker--Planck equation \cref{eq:mf-fp}
and the SDE \cref{eq:mf-sde}.
Our aim is to prove \cref{thm:mf-entropy-convergence}.
To this end, we first show its wellposedness and regularity.

\begin{lem}
\label{lem:wellposedness-regularity}
Suppose \(F\) satisfies \cref{eq:lip}.
Then for every initial value \(m_0\) of finite second moment,
the equation \cref{eq:mf-fp} admits a unique solution in \(C\bigl([0,\infty); \mathcal P_2(\mathbb R^d)\bigr)\).
Moreover, for every \(t > 0\), the measure \(m_t\) is absolutely continuous with respect to the Lebesgue measure.
\end{lem}

\begin{proof}
Since the drift \(D_m F(\cdot, \cdot)\) of the SDE system \cref{eq:mf-sde}
is jointly Lipschitz in measure and in space by our condition \cref{eq:lip},
the existence and uniqueness of the solution is standard.

To show the existence of density we recall Kolmogorov's fundamental solution
\begin{multline*}
\rho_t(x,v; x',v')
\coloneqq \biggl( \frac{\sqrt 3}{2\pi t^2} \biggl)^{d}
\exp \biggl( - \frac{3\lvert x - (x' + tv')\rvert^2}{t^3} \\
+ \frac{3\bigl(x - (x' + tv')\bigr) \cdot (v - v')}{t^2} - \frac{|v - v'|^2}{t}
\biggr)
\end{multline*}
associated to the differential operator
\(\partial_t - \Delta_v + v \cdot \nabla_x\).
Then the Duhamel's formula holds in the sense of distributions:
\begin{multline}
\label{eq:mf-duhamel}
m_t = \int \rho_t(\cdot; z') m_0(dz') \\
+ \int_0^t \iint \rho_{s} (\cdot ; x',v')
\nabla_{v'} \cdot \Bigl( m_{t-s}(dx'dv') \bigl(v' + D_m F(m^x_{t-s}, x')\bigr)\Bigr) ds.
\end{multline}
Since the first moment of \(m_t\) is bounded,
that is, for every \(T > 0\),
\(\sup_{t \in [0,T]} \int (|v| + |x|)m_t(dxdv) < +\infty\),
we can integrate by parts in the second term of \cref{eq:mf-duhamel}
and obtain
\[
\Vert m_t\Vert_{L^1} \leq 1
+ C \int_0^t \sup_{x',v'}\lVert\nabla_{v'} \rho_s(\cdot; x',v')\rVert_{L^1} ds.
\]
By explicit computations we have
\(
\sup_{x',v'}\lVert\nabla_{v'} \rho_s(\cdot; x',v')\rVert_{L^1} = O(s^{-1/2})
\),
from which the existence of the density follows.
\end{proof}

We now introduce a technical condition on the mean field functional:
the mapping \(x \mapsto D_m F(m,x)\) is fourth-order differentiable
with derivatives continuous in measure and in space,
and satisfying
\begin{equation}
\label{eq:space-higher}
\sup_{m \in \mathcal P_2(\mathbb R^d)} \sup_{x \in \mathbb R^d}
\bigl|\nabla^k D_m F(m,x)\bigr| < +\infty,\qquad
\text{$k = 2$, $3$, $4$.}
\end{equation}
This condition will be used to derive some intermediate results
in the following studies of the mean field dynamics.

\begin{defn}[Standard algebra]
\label{defn:std-alg}
We define the \emph{standard algebra} \(\mathcal A_+\) to be the set of
\(C^4\) functions \(h : \mathbb R^{2d} \to (0,\infty)\)
for which there exists a constant \(C\) such that
\[
\lvert\log h(x,v)\rvert \leq C(1 + |x| + |v|)
\quad\text{and}\quad
\sum_{k=1}^4\bigl\lvert\nabla^k h(x,v)\bigr\rvert
\leq \exp\bigl(C(1+|x|+|v|)\bigr)
\]
holds for every \((x, v) \in \mathbb R^{2d}\).
For a collection of functions \((h_\iota)_{\iota \in I}\) we say
\(h_\iota \in \mathcal A_+\) \emph{uniformly} for \(\iota \in I\)
or \((h_\iota)_{\iota \in I} \subset \mathcal A_+\) \emph{uniformly},
if there exists a constant \(C\) such that
the previous bounds holds for every \(h_\iota\), \(\iota \in I\).
\end{defn}

\begin{prop}[Density of \(\mathcal A_+\)]
\label{prop:std-alg-density}
Assume \(F\) satisfies \cref{eq:lip,eq:space-higher}
and there exists a measure \(m_\infty\) satisfying \cref{eq:mf-invariant-measure}
and having finite exponential moments.
Then for every \(m \in \mathcal P_2(\mathbb R^d)\) with finite entropy and finite Fisher information,
there exists a sequence of measures \((m_n)_{n \in \mathbb N}\) such that
\(m_n / m_\infty \in \mathcal A_+\)
and
\[
W_2 (m_n, m) \to 0,\quad
H(m_n) \to H(m),\quad
I(m_n) \to I(m)
\]
when \(n \to +\infty\).
\end{prop}

\begin{proof}
Let \(\varepsilon\) be arbitrary positive real.
Put \(h = m / m_\infty\).
Define \(h'_n = (h \wedge n) \vee \frac 1n\)
and the associated probability measure
\(m'_n = h'_n m_\infty \big/ \!\int h'_n m_\infty\).
Let \(N \in \mathbb N\) be big enough so that \(\int h'_N m_\infty > 0\).
Note that
\[
\sup_{n \geq N} |x|^2 m'_n(x) \leq \frac{|x|^2 m}{\int h'_N m_\infty},
\]
that is to say, the second moments of \((m'_n)_{n\in\mathbb N}\) are uniformly bounded.
Together with the fact that the density of \(m'_n\) converges to that of \(m\) pointwise,
we have \(m'_n \to m\) in \(\mathcal P_2\).
By the dominated convergence theorem, the sequence of measures \(m'_n\) satisfies
\begin{align*}
H(m'_n)
= \frac{\int \log (h'_n m_\infty) h'_n m_\infty}{\int h'_n m_\infty} - \log \int h_n' m_\infty \to \int m \log m, \quad\text{when \(n \to +\infty\).}
\end{align*}
Moreover,
we have the convergence of Fisher information as
\[
\int \frac{|\nabla (h'_n m_\infty)|^2}{h'_n m_\infty}
= \int \biggl[ \frac{|\nabla m_\infty|^2 h'_n}{m_\infty}
+ \biggl( 2 \frac{\nabla h \cdot \nabla m_\infty}{h m_\infty}
+ \frac{|\nabla h|^2 m_\infty}{h} \biggr)
\mathbf 1_{1/n \leq h \leq n} \biggr]
\]
converges to \(I(m)\) when \(n \to +\infty\),
where we used the fact that the weak derivatives satisfy
\(\nabla h'_n = \nabla h \mathbf 1_{1/n \leq h \leq n}\).
Hence we may choose \(n_0 \in \mathbb N\) such that
\[
W_2(m_{n_0}, m) + \bigl|H(m_{n_0}) - H(m)\bigr|
+ \lvert I(m_{n_0}) - I(m)\rvert \leq \frac{\varepsilon}{2}.
\]
Now set \(m''_n = m'_{n_0} \star \eta_{n}\),
where \((\eta_n)_{n \in \mathbb N}\) is a sequence of smooth mollifiers supported in the unit ball.
We have \(m''_n \to m'_{n_0}\) in \(\mathcal P_2\).
By the convexity of entropy and Fisher information we have
\(H(m''_n) \leq H\bigl(m'_{n_0}\bigr)\)
and \(I(m''_n) \leq I\bigl(m'_{n_0}\bigr)\),
and by the lower semicontinuities in \cref{lem:lower-semi-cont} we have
\(\liminf_{n \to +\infty} H(m''_n) \geq H\bigl(m'_{n_0}\bigr)\)
and
\(\liminf_{n \to +\infty} I(m''_n) \geq I\bigl(m'_{n_0}\bigr)\).
Hence,
\[
W_2 \bigl(m''_n, m'_{n_0}\bigr)
+ \bigl|H\bigl(m''_n\bigr) - H\bigl(m'_{n_0}\bigr)\bigr|
+ \bigl|I\bigl(m''_n\bigr) - I\bigl(m'_{n_0}\bigr)\bigr| \to 0
\]
when \(n \to +\infty\).
So we pick another \(n_1 \in \mathbb N\) such that
\(W_2\bigl(m''_{n_1}, m'_{n_0}\bigr)
+ \bigl|H\bigl(m''_{n_1}\bigr) - H\bigl(m'_{n_0}\bigr)\bigr|
+ \bigl|I\bigl(m''_{n_1}\bigr) - I\bigl(m'_{n_0}\bigr)\bigr|
< \varepsilon/2\).

It remains to verify that \(m''_{n_1} \big/ m_\infty\) belongs to \(\mathcal A_+\).
By the definition we have
\[
\frac{m''_{n_1}}{m_\infty} = \frac{(h'' m_\infty) \star \rho_{n_1}}{m_\infty}
\]
for some \(h''\) with \(0 < \inf h'' \leq \sup h'' < +\infty\).
Hence for every \(z \in \mathbb R^{2d}\),
\[
\inf h'' \frac{\inf_{B(z,1)} m_\infty}{m_\infty(z)}
\leq \frac{m''_{n_1}(z)}{m_\infty(z)} \leq \sup h'' \frac{\sup_{B(z,1)} m_\infty}{m_\infty(z)}.
\]
On the other hand, the gradient of \(m_\infty\) satisfies
\(\lvert\nabla \log m_\infty(z)\rvert
\leq \lvert D_m F(m_\infty,x)\rvert + |v|\)
for every \(z = (x,v) \in \mathbb R^{2d}\).
In particular, we have
\[
\exp\bigl(-C(1+|x| + |v|)\bigr) \leq
\frac{\inf_{B(z,1)} m_\infty}{m_\infty(z)} \leq
\frac{\sup_{B(z,1)} m_\infty}{m_\infty(z)}
\leq \exp\bigl(C(1+|x| + |v|)\bigr),
\]
for some constant \(C\).
Therefore, \(m''_{n_1} \big/ m_\infty\) verifies the first condition of \(\mathcal A_+\).

Now verify the conditions on the derivatives.
The derivatives read
\[
\nabla^k \biggl( \frac{m''_{n_1}}{m_\infty} \biggr)
= \sum_{j=0}^k \binom kj \frac{(h'' m_\infty) \star \nabla^j \rho_{n_1}}{m_\infty}
\cdot m_\infty \nabla^{k-j} \bigl(m_\infty^{-1}\bigr).
\]
For each term in the sum, we can bound its first part by
\[
\biggl|\frac{(h'' m_\infty) \star \nabla^j \rho_{n_1} (z)}{m_\infty(z)}\biggr|
\leq \exp\bigl(C(1+|z|)\bigr),
\]
using the same method that we used to verify
the first condition of \(\mathcal A_+\).
Moreover, since our assumptions \cref{eq:lip,eq:space-higher} imply
\[
\lvert\nabla \log m_\infty (z)\rvert \leq C\bigl(1 + |z|\bigr)\qquad\text{and}\qquad
\bigl|\nabla^k \log m_\infty(z)\bigr| \leq C\quad
\text{for $k = 2$, $3$, $4$,}
\]
the second part of each term of the sum,
\(m_\infty \nabla^{k-j} \bigl(m_\infty^{-1}\bigr)\), is of polynomial growth.
The proof is then complete.
\end{proof}

Then we show the stability of the set \(\mathcal A_+\)
under the mean field flow.
This property will be used to justify the computations in the proof
of \cref{thm:mf-entropy-convergence},
as is usual in the analysis of PDE.

\begin{prop}[Stability of \(\mathcal A_+\) under flow]
\label{prop:std-alg-stability}
Assume that \(F\) satisfies \cref{eq:lip,eq:space-higher}
and that there exists a measure \(m_\infty\) satisfying \cref{eq:mf-invariant-measure}
and having finite exponential moments.
Let \((m_t)_{t \in [0,T]} \in C([0,T]; \mathcal P_2(\mathbb R^d))\) be a solution
in the sense of distributions
to the mean field Fokker--Planck equation \cref{eq:mf-fp}.
If \(m_0 / m_\infty \in \mathcal A_+\),
then \(m_t / m_\infty \in \mathcal A_+\) uniformly for \(t \in [0,T]\).
In particular, \(m_t\) is a classical solution to  the Fokker--Planck equation.
\end{prop}

\begin{proof}
In the following \(C\) will denote a constant depending on \(M^F_{mx}\),
\(M^F_{mm}\), the initial value \(h_0 = h(0,\cdot) \coloneqq m_0 / m_\infty\),
the time interval \(T\) and the bounds on the higher-order derivatives
\(\max_{k=2,3,4}\sup_{m,x} \bigl|\nabla^k D_mF(m,x)\bigr|\),
and it may change from line to line.
For a given quantity \(Q\),
we denote by \(C_Q\) a constant depending additionally on \(Q\).

Denote \(b_t(x) = - D_m F(m_t,x)\) and \(b_\infty(x) = - D_m F(m_\infty,x)\).
We also define \(h_t (x) = m_t(x) / m_\infty(x)\).
The relative density solves
\begin{equation}
\label{eq:h}
\partial_t h = \Delta_v h - v \cdot \nabla_v h - v \cdot \nabla_x h - b_t \cdot \nabla_v h + (b_t - b_\infty) \cdot v h.
\end{equation}
Fix \(t \in [0,T]\).
We construct
for every \(z = (x,v) \in \mathbb R^{2d}\),
the stochastic process \(Z^{t,z}_s = (X^{t,z}, V^{t,z})\), solving
\begin{align*}
dX^{t,z}_s &= - V^{t,z}_s ds, \\
dV^{t,z}_s &= - V^{t,z}_s ds - b_{t-s}\bigl(X^{t,z}_s\bigr) ds + \sqrt 2 dW_s
\end{align*}
for \(s \in [0,t]\)
with the initial value \(X^{t,z}_0 = x, V^{t,z}_0 = v\)
and the \emph{same} Brownian motion \((W_s)_{s \in [0,t]}\).

\proofstep{Regularity of \(Z^{t,z}_s\)}
Set \(M^{t,z} = \sup_{s \in [0,t]}\bigl| Z^{t,z}_s\bigr|\).
By Itō's formula and Doob's maximal inequality, the processes satisfy
for every \(\alpha \geq 0\),
\begin{equation}
\label{eq:M-exponential-moments}
\Expect \bigl[\exp (\alpha M^{t,z})\bigr]
\leq \exp \bigl(C_\alpha (1 + |z|)\bigr).
\end{equation}
Thanks to the assumption on the uniform boundedness of the higher-order derivatives
\cref{eq:space-higher},
the mapping \(z \mapsto Z^{t,z}_s\) is \(C^4\) and the partial derivatives solve the Cauchy--Lipschitz SDEs
for \(k = 1\), \(2\), \(3\), \(4\):
\begin{align*}
d \nabla^k X^{t,z}_s &= - \nabla^k V^{t,z}_s ds, \\
d \nabla^k V^{t,z}_s &= - \nabla^k V^{t,z}_s ds
- \sum_{j=1}^k \nabla^j b_{t-s} \bigl(X^{t,z}_s\bigr)
B_{k,j}\bigl(\nabla X^{t,z}_s,\ldots,\nabla^{k-j+1} X^{t,z}_s\bigr) ds,
\end{align*}
where \(B_{k,j}\) is a \(k - j + 1\)-variate polynomial
and in particular \(B_{k,1}(x_1,\ldots,x_k) = x_k\).
The initial values of the SDEs read
\[
\nabla Z^{t,z}_0 = \Id\qquad\text{and}\qquad \nabla^k Z^{t,z}_0 = 0\quad
\text{for $k=2$, $3$, $4$.}
\]
By induction we can obtain the almost sure bound
\begin{equation}
\label{eq:Z-bound}
\max_{k = 1,2,3,4} \sup_{s \in [0,t]}\bigl|\nabla^k Z^{t,z}_s\bigr| \leq C.
\end{equation}

\proofstep{Regularity of \(h\) by Feynman--Kac}
Denote \(g(t,z) = g(t,x,v) = \bigl(b_t(x) - b_\infty(x)\bigr) \cdot v\).
It satisfies
\[
\bigl|g(t,z)\bigr| \leq M^F_{mm} W_2(m_t, m_\infty) |v| \leq M^F_{mm} \sup_{t \in [0,T] } W_2(m_t, m_\infty) |v| = C|v|
\]
and also \(\bigl|\nabla^k g(t,z)\bigr| \leq C\bigl(1+|z|\bigr)\)
for \(k = 1\), \(2\), \(3\), \(4\).
The Feynman--Kac formula for the parabolic equation \cref{eq:h} reads
\begin{equation}
\label{eq:h-feynman-kac}
h(t,z) = \Expect \biggl[ \exp \biggl( \int_0^t
g\bigl(t-s,Z^{t,z}_s\bigr) ds \biggr)
h\bigl(0,Z^{t,z}_t\bigr) \biggr].
\end{equation}
Using the method in the proof of \cite[Proposition 4.12]{ulpoc}, we can prove
\[
\lvert \log h(t,z)\rvert \leq C(1+|z|).
\]
Moreover, thanks to the estimates \cref{eq:Z-bound},
we can apply the dominated convergence theorem to
the Feynman--Kac formula \cref{eq:h-feynman-kac}
and obtain that \(z \mapsto h(t,z)\) belongs to \(C^4\)
with partial derivatives
\begin{multline*}
\nabla^k h(t,z) = \sum_{j=0}^k \Expect \biggl[
\exp \biggl( \int_0^t g\bigl(t-s,Z^{t,z}_s\bigr) ds \biggr)
P_j \biggl( \int_0^t \nabla_z g\bigl(t-s, Z^{t,z}_s\bigr) ds,
\ldots, \\
\int_0^t \nabla_z^j g\bigl(t-s, Z^{t,z}_s\bigr) ds \biggr)
\nabla^{k-j}_z h\bigl(0,Z^{t,z}_t\bigr)
\biggr],
\end{multline*}
where \(P_j\) is a \(j\)-variate polynomial.
Note that
\begin{align*}
\nabla^{k}_z f(Z^{t,z}_s) &=
\sum_{\ell = 1}^{k} \nabla^{\ell} f\bigl(Z^{t,z}_s\bigr)
B_{k,\ell}\bigl(\nabla Z^{t,z}_s, \ldots, \nabla^{k-\ell+1} Z^{t,z}_s\bigr)
\end{align*}
holds for \(f = g(t-s,\cdot), s \in [0,t]\) and for \(f = h(0,\cdot)\).
We apply the bounds on \(\bigl|\nabla^k g\bigr|, \bigl|\nabla^k h\bigr|\)
for \(k=0\), \(1\), \(2\), \(3\), \(4\)
and the exponential moment bound \cref{eq:M-exponential-moments}
to obtain that
\(\bigl|\nabla^k h(t,z)\bigr| \leq \exp \bigl(C(1+|z|)\bigr)\)
for \(k = 1\), \(2\), \(3\), \(4\).
Finally, the derivatives \(\nabla h, \nabla^2 h\) exist
and one can show that they are continuous in time
by differentiating \cref{eq:h} twice in space.
So again by the equation \cref{eq:h} we have \(\partial_t h\) is continuous
and therefore exists classically.
Thus \(m_t\) is a classical solution to  the Fokker--Planck equation \cref{eq:mf-fp}.
\end{proof}

\begin{rem}
The polynomials appearing in the previous proof
belong to the \emph{non-commutative} free algebras over \(\mathbb R\)
of respective number of indeterminates
instead of the usual polynomial rings,
as the tensor product is not commutative.
\end{rem}

After the technical preparations we prove \cref{thm:mf-entropy-convergence}.

\begin{proof}[Proof of \cref{thm:mf-entropy-convergence}]
Suppose first that
the mean field functional \(F\) satisfies additionally \cref{eq:space-higher}
and the initial value of the dynamics is such that \(m_0 / m_\infty\)
belongs to \(\mathcal A_+\),
which is the standard algebra defined in \cref{defn:std-alg}.
According to \cref{prop:std-alg-stability},
the measure \(m_t\) belongs to \(\mathcal A_+\) uniformly in \(t\),
for every \(T > 0\).
Since we have that \(z \mapsto \hat m_t(z) / m_\infty(z)\) is \(C^4\) with
\begin{align}
\sup_{z \in \mathbb R^{2d}}
\Bigl|\nabla \log \frac{\hat m_t}{m_\infty} (z)\Bigr|
&\leq M^F_{mm}W_2(m_t,m_\infty) \nonumber \\
\intertext{and}
\max_{k=2, 3, 4}\sup_{z \in \mathbb R^{2d}}
\Bigl|\nabla^k \log \frac{\hat m_t}{m_\infty} (z) \Bigr| &\leq M, \nonumber \\
\intertext{for some constant \(M\),
the alternative relative density \(\eta_t(z) \coloneqq m_t(z) / \hat m_t(z)\)
is \(C^4\) in \(z\) and there exists a constant \(M_T\) such that}
\eta_t(z) + \frac{1}{\eta_t(z)}
+ \sum_{k=1}^4 \bigl|\nabla^k \eta_t(z)\bigr|
&\leq \exp \bigl(M_T (1 + |z|)\bigr)
\label{eq:eta-bound}
\end{align}
for every \((t,z) \in [0,T] \times \mathbb R^{2d}\).
The constant \(M_T\) may change from line to line in the following.

In the following we will adopt the abstract notations introduced
by Villani in his seminal work on the hypocoercivity \cite{hypocoer}.
Define \(\mathcal H_t = L^2(\hat m_t)\), \(A_t = \nabla_v\)
and \(B_t = v \cdot \nabla_x - D_m F(m_t, x) \cdot \nabla_v\).
The adjoint of \(A_t\) in \(\mathcal H_t\) is therefore
\(A^*_t = -\nabla_v + v\),
while \(B_t\) is antisymmetric: \(B^*_t = - B_t\).
Define the commutator \(C_t = [A_t,B_t] = A_tB_t - B_tA_t = \nabla_x\).
Finally define \(L_t = A^*_tA_t + B_t\) and \(u_t = \log \eta_t\).
The Fokker--Planck equation \cref{eq:mf-fp} now reads
\begin{equation}
\label{eq:eta-dynamics}
\frac{\partial_t m_t}{\hat m_t} = - L_t \eta_t = - (A^*_t A_t + B_t) \eta_t.
\end{equation}

\proofstep{Adding anisotropic Fisher}
Let \(a,b,c\) be positive reals to be determined.
We define the hypocoercive Lyapunov functional
\begin{multline}
\label{eq:mf-hypocoer-functional}
\mathcal E(m) = \mathcal F(m)
+ a \int \Bigl|\nabla_v \log \frac{m}{\hat m} (z)\Bigr|^2 m(dz)\\
+ 2b \int \nabla_v \log \frac{m}{\hat m} (z) \cdot \nabla_x \log \frac{m}{\hat m}(z) m(dz)
+ c \int \Bigl|\nabla_x \log \frac{m}{\hat m} (z)\Bigr|^2 m(dz),
\end{multline}
where \(\mathcal F(m) = F(m) + \frac 12 \int |v|^2 m + H(m)\)
is the free energy.
We also denote the sum of the last three terms in \cref{eq:mf-hypocoer-functional} by
\(I_{a,b,c} (m_t | \hat m_t)\), so that
\[\mathcal E(m) = \mathcal F(m) + I_{a,b,c} (m_t | \hat m_t).\]
Thanks to \cref{prop:std-alg-stability} and in particular
the bound \cref{eq:eta-bound},
we can show that the quantity \(\mathcal E(m_t)\) is well defined
for every \(t \geq 0\) and is continuous in \(t\).
We will show in the following
that \(t \mapsto \mathcal E(m_t)\) is in fact absolutely continuous
and calculate its almost everywhere derivative.
To this end, for every \(t > 0\) and every \(h \geq -t\), we define
\begin{align*}
\mathcal E(m_{t+h}) - \mathcal E(m_t)
&= \bigl( \mathcal F(m_{t+h}) - \mathcal F(m_t) \bigr) \\
&\quad+ \bigl( I_{a,b,c} (m_{t+h} | \hat m_{t+h}) - I_{a,b,c} (m_t | \hat m_{t+h}) \bigr) \\
&\quad+ \bigl( I_{a,b,c} (m_t | \hat m_{t+h}) - I_{a,b,c} (m_t | \hat m_t) \bigr) \\
&\eqqcolon \Delta_1 + \Delta_2 + \Delta_3.
\end{align*}

\proofstep{Contributions from \(\Delta_1\) and \(\Delta_2\)}
We first calculate the contributions from \(\Delta_1\).
Using the Fokker--Planck equation \cref{eq:mf-fp}
and the bounds \cref{eq:eta-bound},
one has \(|\Delta_1| \leq M_T h\) for every \(t\), \(h\)
such that \(t\) and \(t+h\) belong to \([0,T]\);
moreover, by the dominated convergence theorem
one has for almost every \(t > 0\),
\[
\lim_{h \to 0} \frac{\Delta_1}{h}
= \frac{d\mathcal F(m_t)}{dt}
= - \int \Bigl| \nabla_v \log \frac{m_t}{\hat m_t} (z) \Bigr|^2 m_t(z) dz
= - \int |A_t u_t|^2 m_t,
\]
where the right hand side is continuous in \(t\).
The above inequality then holds for every \(t > 0\).
Define the \(4 \times 4\) matrix
\[
K_1 = \begin{pmatrix}
1 & 0 & 0 & 0 \\
0 & 0 & 0 & 0 \\
0 & 0 & 0 & 0 \\
0 & 0 & 0 & 0
\end{pmatrix},
\]
and denote the Hilbertian norm by \(\Vert \cdot \Vert = \Vert \cdot \Vert_{L^2(m_t)}\).
Introduce the four-dimensional vector
\begin{equation}
\label{eq:Y}
Y_t =
\bigl(\Vert A_tu_t \Vert, \bigl\Vert A^2_tu_t\bigr\Vert,
\Vert C_tu_t\Vert, \Vert C_tA_tu_t\Vert\bigr)^\mathsf{T}.
\end{equation}
Then we have for almost every \(t > 0\),
\(\lim_{h \to 0}\Delta_1/h = - Y_t^\mathsf{T} K_1 Y_t\).

Next calculate the contributions from \(\Delta_2\).
Arguing as we did for \(\Delta_1\),
again we have \(|\Delta_2| \leq M_T h\).
Applying the dominated convergence theorem
and compute as in the proofs of \cite[Lemma 32 and Theorem 18]{hypocoer},
we obtain that
for almost every \(t > 0\),
the limit \(\lim_{h \to 0} \Delta_2/h\) exists
and is upper bounded by \(- Y_t^\mathsf{T} K_2 Y_t \),
where
\[
K_2 \coloneqq \begin{pmatrix}
2a - 2M^F_{mx}b & - 2b & - 2a & 0 \\
0 & 2a & - 2M^F_{mx}c & -4b \\
0 & 0 & 2b & 0 \\
0 & 0 & 0 & 2c
\end{pmatrix}.
\]

\proofstep{Contributions from \(\Delta_3\)}
Now we calculate the last term
\(\Delta_3 \coloneqq I_{a,b,c}(m_t | \hat m_{t+h}) - I_{a,b,c} (m_t | \hat m_t)\).
Note that \(\nabla_v \log \hat m_t(z) = -v\)
and, by the $W_2$-Lipschitz continuity of $m \mapsto D_mF(m,x)$, we have
\begin{multline*}
\bigl|\nabla_x \log \hat m_{t+h}(z)
- \nabla_x \log \hat m_t(z) \bigr|
= \bigl|D_m F\bigl(m^x_{t+h}, x\bigr) - D_m F\bigl(m^x_t, x\bigr)\bigr| \\
\leq M^F_{mm} W_2\bigl(m^x_{t+h}, m^x_t\bigr).
\end{multline*}
So for each $z \in \mathbb R^{2d}$,
we know that \( \nabla \log \hat m_t(z) \) is continuous in $t$,
and is absolutely continuous once
\(t \mapsto m^x_t\) is absolutely continuous with respect to the \(W_2\) distance in the sense of \cite[Definition 1.1.1]{gf}.
Let us show the latter.
Integrating the speed component in the Fokker--Planck equation \cref{eq:mf-fp},
we obtain
\begin{equation}
\label{eq:mf-x-ce}
\partial_t m^x_t + \nabla_x \cdot\bigl(v^x_t m^x_t\bigr) = 0,
\end{equation}
where
\[
v^x_t (x) \coloneqq \frac{\int vm_t(x,v) dv}{\int m_t(x,v) dv}
= \frac{\int \nabla_v \log \frac{m_t}{\hat m_t}(x,v) m_t(x,v) dv}{\int m_t(x,v) dv}
\]
is the average speed at the spatial point \(x\).
The \(L^2\) norm of the vector field in the continuity equation
\cref{eq:mf-x-ce} satisfies
\begin{multline*}
\Vert v^x_t \Vert_{L^2(m^x_t)}
= \Biggl( \int \Biggl|\frac{\int \nabla_v \log \frac{m_t}{\hat m_t}(x,v) m_t(x,v) dv}{\int m_t(x,v) dv}\Biggr|^2 m^x_t(x) dx \Biggr)^{\!1/2} \\
\leq \biggl( \int \Bigl| \nabla_v \log \frac{m_t}{\hat m_t} (z)\Bigr|^2 m_t(dz) \biggr)^{\!1/2}
= \Vert A_t u_t \Vert \leq M_T,
\end{multline*}
where the first inequality is due to Cauchy--Schwarz.
Applying \cite[Proposition 8.3.1]{gf} to the flow \(t \mapsto m^x_t\)
and its continuity equation \cref{eq:mf-x-ce},
and using \cite[Theorem 1.1.2]{gf},
we obtain
\[
W_2\bigl(m^x_{t+h}, m^x_t\bigr)
\leq \int_t^{t+h} \Vert A_s u_s \Vert ds \leq M_T h
\]
for every \(t\), \(h\) such that
\(t\) and \(t + h\) belong to \([0, T]\).
So the mapping \(t \mapsto \nabla \log \hat m_t(z)\)
is absolutely continuous with almost everywhere derivatives satisfying
\begin{align*}
\partial_t \nabla_v \log \hat m_t(z) &= 0, \\
\lvert\partial_t \nabla_x \log \hat m_t (z)\rvert
&\leq M^F_{mm} \Vert A_t u_t \Vert \leq M_T.
\end{align*}
Then we obtain \(|\Delta_3| \leq M_T h\).
Moreover, by the dominated convergence theorem,
we have for almost every \(t > 0\),
\begin{multline*}
\lim_{h \to 0} \frac{|\Delta_3|}{|h|}
\leq 2M^F_{mm}
\int \bigl( \lvert A_t u_t(z)\rvert, \lvert C_t u_t(z)\rvert \bigr)
\begin{pmatrix}
a & b \\
b & c
\end{pmatrix}
\begin{pmatrix}
0 \\
\Vert A_t u_t \Vert
\end{pmatrix}
m_t(dz) \\
\leq 2 M^{F}_{mm} ( b \Vert A_t u_t \Vert \Vert A_t u_t \Vert
+ c \lVert A_t u_t \rVert \lVert C_t u_t \rVert )
= Y_t^\mathsf{T} K_3 Y_t
\end{multline*}
by applying Cauchy--Schwarz again, where
\[
K_3 \coloneqq 2M_{mm}^{F}\begin{pmatrix}
b & 0 & c & 0 \\
0 & 0 & 0 & 0 \\
0 & 0 & 0 & 0 \\
0 & 0 & 0 & 0
\end{pmatrix}.
\]

\proofstep{Hypocoercivity}
Our previous bounds on \(\Delta_k\), \(k = 1\), \(2\), \(3\)
establish that \(t \mapsto \mathcal E(m_t)\) is absolutely continuous
(locally Lipschitz, in fact)
with its almost everywhere derivative satisfying
\(\frac{d}{dt} \mathcal E(m_t) \leq - Y_t^\mathsf{T} K Y_t\),
where \(K\) is defined by \(K_1 + K_2 - K_3\) and is equal to
\[
\begin{pmatrix}
1 + 2M^F_{mm}a - 2\bigl(M^F_{mx} + M^F_{mm}\bigr)b
& - 2b & - 2a - 2M^F_{mm} c & 0 \\
0 & 2a & - 2M^F_{mx}c & -4b \\
0 & 0 & 2b & 0 \\
0 & 0 & 0 & 2c
\end{pmatrix}.
\]
As in the end of the proof of \cite[Theorem 18]{hypocoer},
we can pick constants \(a\), \(b\), \(c > 0\), depending only on \(M^F_{mx}\) and \(M^F_{mm}\),
such that \(ac > b^2\) and the matrix \(K\) is a positive-definite.
Let \(\alpha\) be the smallest eigenvalue of \(K\).
Then we have
\begin{multline*}
\frac{d\mathcal E(m_t)}{dt}
\leq - \alpha \bigl(\Vert A_t u_t \Vert^2 + \Vert C_t u_t \Vert^2 + \Vert A_t^2 u_t \Vert^2 + \Vert C_tA_tu_t\Vert^2\bigr) \\
\leq - \alpha \bigl(\Vert A_t u_t \Vert^2 + \Vert C_t u_t \Vert^2\bigr)
= - \alpha I(m_t | \hat m_t).
\end{multline*}
Hence for every \(t\), \(s\) such that \(t \geq s \geq 0\),
\begin{equation}
\label{eq:mf-hypocoer}
\mathcal E(m_t) \leq \mathcal E(m_s) - \alpha \int_s^t I(m_u | \hat m_u) du.
\end{equation}

\proofstep{Approximation}
We now show that the inequality \cref{eq:mf-hypocoer}
holds without additional assumptions on the mean field functional \(F\)
and the initial value \(m_0\).

First, suppose still that \(F\) satisfies \cref{eq:space-higher}
but no longer suppose \(m_0\) is such that \(m_0 / m_\infty \in \mathcal A_+\).
The initial value \(m_0\) belongs to \(\mathcal P_2(\mathbb R^d)\) and both \(H(m_0)\) and \(I(m_0)\) are finite,
so thanks to \cref{prop:std-alg-density}, we can pick a sequence of measures
\(\bigl(m'_{n,0}\bigr)_{n\in\mathbb N}\),
each of which belongs to \(\mathcal A_+\), such that
\[
\lim_{n\to\infty} W_2\bigl(m'_{n,0},m_0\bigr)
+ \bigl|H\bigl(m'_{n,0}\bigr) - H(m_0)\bigr|
+ \bigl|I\bigl(m'_{n,0}\bigr) - I(m_0)\bigr| = 0.
\]
As proved above, the inequality \cref{eq:mf-hypocoer} holds
for the flow \(\bigl(m'_{n,t}\bigr)_{t \geq 0}\), that is,
\[
\mathcal E\bigl(m'_{n,t}\bigr)
\leq \mathcal E\bigl(m'_{n,0}\bigr)
- \alpha \int_0^t I\bigl(m'_{n,s} \big| \hat m'_{n,s}\bigr) ds.
\]
By the continuity with respect to the initial value of the SDE system \cref{eq:mf-sde},
we have also \(m'_{n,t} \to m_t\) in the weak topology of \(\mathcal P_2\).
We recall in \cref{lem:lower-semi-cont} that both the entropy and the Fisher information are lower semicontinuous with respect to the weak topology of \(\mathcal P_2\).
Taking the lower limit on both sides of the inequality above, we obtain
\cref{eq:mf-hypocoer} with \(s = 0\) for the original flow \((m_t)_{t \geq 0}\).

Second, we no longer require \(F\) to satisfy \cref{eq:space-higher}
and set \(F_k (m) = F (m \star \rho_k)\)
for a sequence of smooth and symmetric mollifiers \((\rho_k)_{k \in \mathbb N}\) in \(\mathbb R^d\) with \(\supp \rho_k \subset B(0,1/k)\).
The linear derivative of the regularized mean field functional reads
\(
\frac{\delta F_k}{\delta m}(m,\cdot) = \frac{\delta F}{\delta m}(m \star \rho_k,\cdot) \star \rho_k
\),
and its intrinsic derivative reads \(D_m F_k(m, \cdot) = D_m F(m \star \rho_k, \cdot) \star \rho_k\).
Consequently,
\begin{multline}
\label{eq:F_k-approx}
\lvert D_m F_k (m',x') - D_m F(m,x)\rvert
\leq M^F_{mm} W_2(m', m) + M^F_{mx} |x' - x| \\
+ \frac{M^F_{mx} + M^F_{mm}}{k}.
\end{multline}
Moreover, \(\nabla D_m F_k(m, \cdot) = \nabla D_m F(m \star \rho_k, \cdot) \star \rho_k\) and
\[
\nabla^k D_m F_k(m, \cdot)
= D_m F(m \star \rho_k, \cdot) \star \nabla^k \rho_k
= \nabla D_m F(m \star \rho_k, \cdot) \star \nabla^{k-1} \rho_k
\]
is continuous for \(k \geq 0\) and bounded for \(k \geq 1\).
In particular \(F_k\) satisfies \cref{eq:space-higher}.
Define \(\mathcal E_k (m) = F_k(m) + \frac 12 \int |v|^2 m + H(m)
+ I_{a,b,c} (m | \hat m)\)
and here \(\hat m\) should be understood as the Gibbs-type measure
defined with \(F_k\) instead of \(F\).
Let \((m''_{k,t})_{t \geq 0}\) be the flow of measures driven by \(F_k\)
with the initial value \(m''_{k,0} = m_0\).
Our previous result yields for every \(t \geq 0\),
\[
\mathcal E_k\bigl(m''_{k,t}\bigr)
\leq \mathcal E_k(m_{0})
- \alpha \int_0^t I\bigl(m''_{k,s} \big| \hat m''_{k,s}\bigr) ds,
\]
where \(\hat m''_{k,s}\) is the probability measure proportional to
\[
\exp \biggl( - \frac{\delta F_k}{\delta m}\bigl(m''_{k,s}, x\bigr)
- \frac 12 |v|^2 \biggr) dxdv.
\]
From the bound \cref{eq:F_k-approx} we deduce that \(m''_{k,t} \to m_t\)
in \(\mathcal P_2\) for every \(t \geq 0\) by the synchronous coupling result in \cref{lem:synchronous-coupling}.
So taking the lower limit on both sides of the previous inequality,
we obtain the inequality \cref{eq:mf-hypocoer} with \(s = 0\)
holds for general initial values and general mean field functionals.
In particular, for every \(t \geq 0\), the measure \(m_t\) has finite entropy
and finite Fisher information.
Then we apply the same argument to the flow with the initial value \(m_s\)
and obtain the inequality \cref{eq:mf-hypocoer} for general \(s \geq 0\).

\proofstep{Conclusion}
Define the matrix
\[
S = \begin{pmatrix}
a & b \\
b & c
\end{pmatrix}
\]
and denote by \(|S|\) its largest eigenvalue.
The Fisher information satisfies for every \(t \geq 0\),
\begin{align*}
I(m_t | \hat m_t) &= \frac 12 I(m_t | \hat m_t) + \frac 12 I(m_t | \hat m_t) \\
&\geq 2 \rho H(m_t | \hat m_t)
+ \frac{1}{2} I(m_t | \hat m_t) \\
&\geq 2 \rho \bigl(\mathcal F(m_t) - \mathcal F(m_\infty)\bigr)
+ \frac{1}{2 |S|} I_{a,b,c}(m_t | \hat m_t) \\
&\geq \biggl( 2 \rho \wedge \frac{1}{2|S|} \biggr)
\bigl(\mathcal E(m_t) - \mathcal E(m_\infty)\bigr),
\end{align*}
where on the second line we applied the uniform LSI \cref{eq:lsi},
with \(\rho\) defined by \cref{eq:defn-rho},
and on the third line we used \cref{lem:mf-entropy-sandwich},
\(m_\infty = \hat m_\infty\) and \(S \preceq \lambda_2\).
Applying Grönwall's lemma%
\footnote{The mapping \(t \mapsto \mathcal E(m_t)\) is lower semicontinuous
by \cref{lem:lower-semi-cont}
and non-increasing by the inequality \cref{eq:mf-hypocoer}.
So it is càdlàg.
It then suffices to convolute the mapping \(t \mapsto \mathcal E(m_t)\)
by a sequence of mollifiers compactly supported in \((0,1)\),
apply the classical Grönwall's lemma
and take the limit.
} to \cref{eq:mf-hypocoer},
we obtain the desired contractivity \cref{eq:mf-entropy-convergence}
with \(\kappa \coloneqq \alpha \bigl( 2 \rho \wedge (2|S|)^{-1} \bigr)\).
\end{proof}

\begin{rem}
Our \cref{thm:mf-entropy-convergence} can be compared to
\cite[Theorem 56]{hypocoer},
where kinetic mean field Langevin dynamics with two-body interaction are studied
and \(O(t^{-\infty})\) entropic convergence to equilibrium is shown,
under the assumption that the mean field dependence is small.
This restriction is lifted by our method
which leverages the functional convexity.
\end{rem}

\begin{rem}
The regularized energy functional \(F_k\) is such that
\(x \mapsto D_m F_k(m, x)\) has bounded derivatives of every order.
However \(m \mapsto D_m F_k (m,x)\) remains only Lipschitz continuous
and we are not aware of any approximation argument
that allows us to obtain differentiability in the measure argument.
Consequently we use still the result from \cite{gf}
to treat this low regularity.
\end{rem}

\subsection{Particle system}
\label{sec:ps}

In this section we study the system of particles described by
the \emph{linear} Fokker--Planck equation \cref{eq:ps-fp}
and the SDE \cref{eq:ps-sde}.
Note that since the dynamics is linear, its wellposedness is classical
and we omit its proof.

We first show that for our model we can construct hypocoercive functionals
whose constants are independent of the number of particles.

\begin{lem}[Uniform-in-\(N\) hypocoercivity]
\label{lem:ps-unif-hypocoer}
Assume \(F\) satisfies \cref{eq:lip}
and there exists a measure \(m^N_\infty\) satisfying \cref{eq:ps-invariant-measure}
and having finite exponential moments.
Let \(t \mapsto m^N_t\) be a solution to the \(N\)-particle Fokker--Planck equation \cref{eq:ps-fp}
in \(C \bigl([0,T]; \mathcal P_2(\mathbb R^{2dN})\bigr)\)
whose initial value \(m^N_0\) has finite entropy and finite Fisher information.
Then there exist constants \(a\), \(b\), \(c\), \(\alpha > 0\)
depending only on \(M^F_{mx}\) and \(M^F_{mm}\)
such that
\(ac > b^2\) and the functional
\begin{align}
\mathcal E^N(m^N) &\coloneqq
\mathcal F^N(m^N) + I_{a,b,c}\bigl(m^N\big| m^N_\infty\bigr) \nonumber \\
&\coloneqq \mathcal F^N(m^N)
+ \sum_{i=1}^N \biggl( a \int \bigl|\nabla_{v^i} \log h^N (\mathbf z)\bigr|^2
m^N(d\mathbf z) \nonumber \\
&\hphantom{\coloneqq \mathcal F^N(m^N) + \sum_{i=1}^N \biggl(}
+ 2b \int \nabla_{v^i} \log h^N (\mathbf z) \cdot \nabla_{x^i}
\log h^N (\mathbf z) m^N(d\mathbf z) \nonumber \\
&\hphantom{\coloneqq \mathcal F^N(m^N) + \sum_{i=1}^N \biggl(}
+ c \int \bigl|\nabla_{x^i} \log h^N (\mathbf z)\bigr|^2 m^N(d\mathbf z) \biggr),
\label{eq:ps-hypocoer-functional}
\end{align}
where \(h^N \coloneqq m^N\!\big/ m^N_\infty\),
is finite on \(m^N_t\) for \(t > 0\);
moreover, the mapping \(t \mapsto \mathcal E^N\bigl(m^N_t\bigr)\) satisfies
\begin{equation}
\label{eq:ps-unif-hypocoer}
\mathcal E^N\bigl(m^N_t\bigr) \leq \mathcal E^N\bigl(m^N_s\bigr)
- \alpha \int_s^t I\bigl(m^N_u \big| m^N_\infty\bigr) du
\end{equation}
for every \(t\), \(s\) such that \(t \geq s \geq 0\).
\end{lem}

\begin{rem}
The constants \(a\), \(b\), \(c\) are possibly different from those appearing in the proof of \cref{thm:mf-entropy-convergence}.
\end{rem}

\begin{proof}
We first show that the condition \cref{eq:lip} implies
a bound on the second-order derivatives of
\(\mathbf x \mapsto U^N(\mathbf x) \coloneqq NF(\mu_{\mathbf x})\).
The first-order derivatives satisfy
\begin{multline*}
\bigl|\nabla_i U^N (\mathbf x) - \nabla_i U^N (\mathbf x')\bigr|
= \bigl|D_m F(\mu_{\mathbf x}, x^i) - D_m F(\mu_{\mathbf x'}, x'^i)\bigr| \\
\leq M^F_{mm} W_2 (\mu_{\mathbf x}, \mu_{\mathbf x'}) + M^F_{mx} |x^i - x'^i|.
\end{multline*}
Summing over \(i\), we obtain for every \(\varepsilon > 0\),
\begin{multline*}
\bigl| \nabla U^N (\mathbf x) - \nabla U^N (\mathbf x') \bigr|^2
\leq (1 + \varepsilon)\bigl(M^F_{mm}\bigr)^2 NW_2^2(\mu_{\mathbf x}, \mu_{\mathbf x'})
+ (1 + \varepsilon^{-1}) \bigl(M^F_{mx}\bigr)^2 |\mathbf x - \mathbf x'|^2 \\
\leq \Bigl((1 + \varepsilon)\bigl(M^F_{mm}\bigr)^2
+ (1 + \varepsilon^{-1}) \bigl(M^F_{mx}\bigr)^2 \Bigr)
|\mathbf x - \mathbf x'|^2.
\end{multline*}
Optimizing \(\varepsilon\)
yields \(\bigl| \nabla U^N (\mathbf x) - \nabla U^N (\mathbf x') \bigr|
\leq \bigl(M^F_{mm} + M^F_{mx}\bigr) |\mathbf x - \mathbf x'|\).
Define
\[
\bigl\Vert\nabla^2 U^N\bigr\Vert_{\infty}
= \bigl\Vert\nabla^2 U^N\bigr\Vert_{\textnormal{op},\infty}
= \esssup_{\mathbf x \in \mathbb R^{dN}} \sup_{\mathbf x' \in \mathbb R^{dN} : |\mathbf x'|_2 = 1} \bigl|\nabla^2 U^N (\mathbf x) \mathbf x'\bigr|_2.
\]
From the Lipschitz bound we obtain
\begin{equation}
\label{eq:D^2-U^N-bound}
\bigl\Vert\nabla^2 U^N\bigr\Vert_{\textnormal{op},\infty} \leq M^F_{mx} + M^F_{mm}.
\end{equation}

Now suppose there exist a constant \(M\) such that
\(U^N\) satisfies
\begin{equation}
\label{eq:U^N-bound}
\text{$\mathbf x \mapsto U^N (\mathbf x)$ is $C^4$}
\qquad\text{and}\qquad
\sum_{k=3}^4 \bigl\Vert\nabla^k U^N\bigr\Vert_\infty \leq M,
\end{equation}
and that \(h^N_0 = m^N_0\!\big/ m^N_\infty\) satisfies
\begin{equation}
\label{eq:h^N_0-bound}
h^N_0(\mathbf z) + \frac{1}{h^N_0(\mathbf z)} + \sum_{k=1}^4 \bigl|\nabla^k h^N_t(\mathbf z)\bigr| \leq M
\end{equation}
for every \(\mathbf z \in \mathbb R^{2dN}\).
We apply \cref{prop:std-alg-stability}
to show that under our assumptions, there exists a constant \(M_T\) such that
\begin{equation}
\label{eq:h^N-bound}
h^N_t(\mathbf z) + \frac{1}{h^N_t(\mathbf z)} + \sum_{k=1}^4 \bigl|\nabla^k h^N_t(\mathbf z)\bigr|
\leq \exp \bigl( M_T(1 + |\mathbf z|) \bigr)
\end{equation}
for every \((t,\mathbf z) \in [0,T] \times \mathbb R^{2dN}\)
(in fact, \(\mathbf z \mapsto h^N_t(\mathbf z)\) remains lower and upper bounded and its up-to-fourth-order derivatives grow at most polynomially).

We denote \(u^N_t = \log h^N_t\).
In view of the regularity bound \cref{eq:h^N-bound},
we have
\begin{align*}
- \frac{d\mathcal F^N\bigl(m^N_t\bigr)}{dt}
&= \sum_{i=1}^N \int |\nabla_{v^i} u^N_t |^2 m^N_t, \\
- \frac{d}{dt} \int \bigl|\nabla_{v^i} u^N_t \bigr|^2 m^N_t
&= 2 \int \Bigl( \nabla_{x^i} u^N_t \cdot \nabla_{v^i} u^N_t
+ \bigl| \nabla^2_{v^i} u^N_t \bigr|^2 \\
&\hphantom{=2\int\Bigl(}+ \nabla_{v^i} u^N_t \cdot \nabla_{v^i} u^N_t \Bigr) m^N_t, \\
- \frac{d}{dt} \int \nabla_{v^i} u^N_t \cdot \nabla_{x^i} u^N_t m^N_t
&= \int \biggl( - \sum_{j=1}^N \nabla_{v^i} u^N_t \nabla^2_{ij} U^N \nabla_{v^j} u^N_t
+ \bigl|\nabla_{x^i} u^N_t\bigr|^2 \\
&\phantom{=\int\biggl(}+ 2 \nabla^2_{v^i} u^N_t \cdot \nabla_{v^i}\nabla_{x^i} u^N_t
+ \nabla_{v^i} u^N_t \cdot \nabla_{x^i} u^N_t \biggr) m^N_t, \\
- \frac{d}{dt} \int \nabla_{x^i} u^N_t \cdot \nabla_{x^i} u^N_t m^N_t
&= \int \biggl( - 2 \sum_{j=1}^N \nabla_{x^i} u^N_t \nabla^2_{ij} U^N \nabla_{v^j} \nabla_{v^j} u^N_t \\
&\phantom{=\int\biggl(}+ 2 \bigl|\nabla_{x^i}\nabla_{v^i} u^N_t\bigr|^2 \biggr) m^N_t,
\end{align*}
as is computed in \cite{hypocoer}.
Denote the Hilbertian norm by
\(\lVert \cdot \rVert = \lVert \cdot \rVert_{L^2(m^N_t)}\)
and define the four-dimensional vector
\begin{equation}
\label{eq:Y^N}
Y^N_t = \bigl(\bigl\Vert\nabla_{\mathbf v} u^N_t\bigr\Vert,
\bigl\Vert\nabla^2_{\mathbf v} u^N_t\bigr\Vert,
\bigl\Vert\nabla_{\mathbf x}u^N_t\bigr\Vert,
\bigl\Vert\nabla_{\mathbf x}\nabla_{\mathbf{v}} u^N_t\bigr\Vert\bigr)^\mathsf{T}.
\end{equation}
By Cauchy--Schwarz we have
\(- \frac{d}{dt} \mathcal E^N\bigl(m^N_t\bigr) \geq (Y^N_t)^\mathsf{T} K Y^N_t\)
where
\[
K \coloneqq \begin{pmatrix}
1 + 2a - 2\bigl\Vert\nabla^2 U^N\bigr\Vert_{\textnormal{op},\infty} b
& -2b & -2a & 0\\
  & 2a & -2 \bigl\Vert\nabla^2 U^N\bigr\Vert_{\textnormal{op},\infty} c & -4b \\
  & & 2b & 0\\
  & & & 2c
\end{pmatrix},
\]
where \( \Vert\nabla^2 U^N\Vert_{\textnormal{op},\infty} \)
is bounded by \cref{eq:D^2-U^N-bound}.
We then apply the same argument
as in the proof of \cref{thm:mf-entropy-convergence}
to pick \(a\), \(b\), \(c\) such that
\(ac > b^2\) and
\(K\) is positive-definite with its smallest eigenvalue \(\alpha > 0\).
Then,
\[
- \frac{d\mathcal E^N\bigl(m^N_t\bigr)}{dt}
\geq (Y^N_t)^\mathsf{T} K Y^N_t \geq \alpha I\bigl(m^N_t \big| m^N_\infty\bigr),
\]
from which the desired inequality \cref{eq:ps-unif-hypocoer} follows.

We then show the inequality \cref{eq:ps-unif-hypocoer} holds for
general mean field functional \(F\) and initial value \(m^N_0\).
First, suppose still that
\(U^N\) satisfies additionally the bound \cref{eq:U^N-bound}
but no longer suppose \(m^N_0\) satisfies additionally \cref{eq:h^N_0-bound}.
As \(m^N_0\) has finite second moment, finite entropy and finite Fisher information,
we can find a sequence of measures \((m'^N_{n,0})_{n \in \mathbb N}\),
each of which satisfies the bound \cref{eq:h^N_0-bound}, such that
\[
\lim_{n \to +\infty}
W_2\bigl(m'^N_{n,0}, m^N_0\bigr)
+ \bigl|H\bigl(m'^N_{n,0}\bigr) - H\bigl(m'^N_0\bigr)\bigr|
+ \bigl|I\bigl(m'^N_{n,0}\bigr) - I\bigl(m'^N_0\bigr)\bigr|
= 0,
\]
by the procedure in the proof of \cref{prop:std-alg-density}.
We have the convergence \(m'^N_{n,t} \to m^N_t\) in \(\mathcal P_2\).
So taking the lower limit on both sides of
\[
\mathcal E^N\bigl(m'^N_{n,t}\bigr) - \mathcal E^N\bigl(m'^N_{n,0}\bigr)
+ \alpha \int_0^t I\bigl(m'^N_{n,s} \big| m^N_\infty\bigr) ds \leq 0
\]
yields \cref{eq:ps-unif-hypocoer} for \(s = 0\),
thanks to the \(\mathcal P_2\)-continuity of \(F\)
and the \(\mathcal P_2\)-lower-semicontinuity of entropy and Fisher information,
proved in \cref{lem:lower-semi-cont}.

Second, we no longer suppose \(U^N\) satisfies the bound \cref{eq:U^N-bound}
and set
\[
U^{N}_k = U^N \star \rho_k
\]
for a sequence of smooth mollifiers \((\rho_k)_{k \in \mathbb N}\) in \(\mathbb R^{dN}\).
Then \(U^N_k\) is \(C^4\) and satisfies
its second and fourth-order derivatives
\(\nabla^\nu U^N_k = \nabla^2 U^N \star \nabla^{\nu-2}\rho_k\)
with \(\nu = 3\), \(4\) are bounded
as \(\bigl\Vert \nabla^2 U^N \bigr\Vert_\infty \leq M^F_{mx} + M^F_{mm}\).
Moreover, from the bound \cref{eq:D^2-U^N-bound} on \(\nabla^2 U^N\) we deduce
\begin{equation}
\label{eq:U^N_k-approx}
\bigl\Vert \nabla (U^N_k - U^N) \bigr\Vert_\infty \to 0
\end{equation}
and \(\bigl\Vert \nabla^2 U^N_k \bigr\Vert_\infty
\leq \bigl\Vert \nabla^2 U^N \bigr\Vert_\infty \leq M^F_{mx} + M^F_{mm}\).
Let \(\bigl(m''^N_{k,t}\bigr)_{t \geq 0}\) be the flow of measures driven by the regularized potential \(U^N_k\)
with the initial value \(m''^N_{k,0} = m^N_0\)
and denote its invariant measure by \(m''^N_{k,\infty}\).
That is to say,
\(m''^N_{k,\infty}\) is the probability measure
proportional to
\(\exp \bigl( - U^N_k(\mathbf x) - \frac 12 |\mathbf v|^2 \bigr)
d\mathbf xd\mathbf v\).
Thanks to the bound \cref{eq:U^N_k-approx},
we can apply the synchronous coupling result in \cref{lem:synchronous-coupling}
and obtain \(m''^N_{k,t} \to m^N_t\) in \(\mathcal P_2\) for every \(t \geq 0\).
The result obtained in the previous paragraph writes
\begin{multline*}
H\bigl(m''^N_{k,t}\big| m''^N_{k,\infty}\bigr)
+ I_{a,b,c} \bigl(m''^N_{k,t} \big| m''^N_{k,\infty}\bigr)
- H\bigl(m^N_0\big| m''^N_{k,\infty}\bigr)
- I_{a,b,c}\bigl(m^N_0\big| m''^N_{k,\infty}\bigr) \\
+ \alpha \int_0^t I\bigl(m''^N_{k,s}\big| m''^N_{k,\infty}\bigr) ds \leq 0
\end{multline*}
for every \(t \geq 0\)
and we take the lower limit on both sides to obtain \cref{eq:ps-unif-hypocoer}
with \(s = 0\) for general initial values and general mean field functional.
In particular, this implies for every \(t \geq 0\),
\(m^N_t\) has finite entropy and finite Fisher information.
Then we apply the same argument to the flow with \(m^N_s\) as the initial value
and obtain \cref{eq:ps-unif-hypocoer} for general \(s \geq 0\).
\end{proof}

\begin{rem}
\label{rem:uniform-in-N-LSI}
If we additionally assume a uniform-in-\(N\) LSI for \(m^N_\infty\),
then we can directly establish
\[
\frac{d\mathcal E_N\bigl(m_t^N\bigr)}{dt} \leq - \kappa \mathcal E_N\bigl(m_t^N\bigr),
\]
for a constant \(\kappa > 0\) independent of \(N\).
This approach has been explored in a number of previous works.
We do not impose such an assumption or sufficient conditions for it,
as they often requires the mean field interaction
to be small enough or (semi-)convex enough,
excluding the application to neural networks in \cref{sec:app}.
\end{rem}

We then give the proof of \cref{thm:ps-entropy-convergence}.
The method of proof is similar to \cite[Theorem 2.3]{ulpoc}
and we only need to take into account of the additional kinetic terms.
We give a complete proof only for the sake of self-containedness.

\begin{proof}[Proof of \cref{thm:ps-entropy-convergence}]

We pick the positive constants \(a\), \(b\), \(c\), \(\alpha\)
depending only on \(M^F_{mx}\) and \(M^F_{mm}\) such that
\(ac > b^2\) and \cref{eq:ps-unif-hypocoer} holds for every \(t \geq 0\),
according to \cref{lem:ps-unif-hypocoer}.
Then, as in the proof of \cite[Theorem 2.3]{ulpoc},
we will establish a lower bound of the relative Fisher information
\(I_t \coloneqq I\bigl(m^N_t \big| m^N_\infty\bigr)\)
in order to obtain the desired result.

\proofstep{Regularity of conditional distribution}
By local hypoelliptic positivity (see e.g.\ \cite[Theorem A.19 and Corollary A.21]{hypocoer}),
we know that for every \(t > 0\) and every \(\mathbf z \in \mathbb R^{2dN}\),
\(m^N_t(\mathbf z) > 0\).
Let \(i \in \{1,\ldots,N\}\).
Define the marginal density \(m^{N,-i}_t(\mathbf z^{-i}) = \int m^N_t(\mathbf z) d z^i\), which is strictly positive by the local positivity of \(m^N_t\)
and is lower semicontinuous by Fatou's lemma.
By the Fubini theorem,
we have \(\int m^{N,-i}_t (\mathbf z^{-i}) d\mathbf z^{-i} = 1\).
Together with the lower semicontinuity,
we obtain that \(m^{N,-i}_t(\mathbf z^{-i})\) is finite everywhere.
We are therefore able to define the conditional probability density
\[
m^{N,i|-i}_t (z^i | \mathbf z^{-i} )
= \frac{m^{N}_t (\mathbf z)}{m^{N,-i}_t(\mathbf z^{-i})}
= \frac{m^{N}_t (\mathbf z)}{\int m^{N}_t (\mathbf z) dz^{i}},
\]
which is weakly differentiable in \(z^i\) and strictly positive everywhere.
We can also define the conditional density \(m^{N,i|-i}_\infty\) for the invariant measure \(m^{N}_\infty\),
and the regularity follows directly from its explicit expression.

\proofstep{Decomposing Fisher componentwise}
Using the conditional distributions, we can decompose the relative Fisher information as
\begin{align*}
I_t
&= \int \biggl| \nabla \log \frac{m^N_t(\mathbf z)}{m^N_\infty(\mathbf z)}\biggr|^2 m^N_t (d\mathbf z)
= \Expect \Biggl[ \biggl| \nabla \log \frac{m^N_t (\mathbf Z_t)}{m^N_\infty(\mathbf Z_t)}\biggr|^2 \Biggr] \\
&= \sum_{i=1}^N \Expect \Biggl[ \biggl| \nabla_{z^i} \log \frac{m^{N,i|-i}_t
\bigl(Z^i_t\big| \mathbf Z^{-i}_t\bigr) m^{N,-i}_t\bigl(\mathbf Z^{-i}_t\bigr)}
{m^N_\infty (\mathbf Z_t)}\biggr|^2 \Biggr] \\
&= \sum_{i=1}^N \Expect \Biggl[ \biggl| \nabla_{z^i}
\log \frac{m^{N,i|-i}_t\bigl(Z^i_t \big| \mathbf Z^{-i}_t\bigr)}
{m^{N}_\infty (\mathbf Z_t)}\biggr|^2 \Biggr] \\
&
= \sum_{i=1}^N \Expect \biggl[
\Bigl| \nabla_{z^i} \log m^{N,i|-i}_t\bigl(Z^i_t\big|\mathbf Z^{-i}_t\bigr)
+ D_m F\bigl(\mu_{\mathbf X_t}, X^i_t\bigr) + V^i_t \Bigr|^2 \biggr].
\end{align*}

\proofstep{Change of empirical measure and componentwise LSI}
We replace the empirical measure \(\mu_{\mathbf x}\) in \(D_m F\) by \(\mu_{\mathbf x^{-i}}\).
Define the difference \(\delta^i_1 (\mathbf x; y) = D_m F(\mu_{\mathbf x}, y) - D_m F(\mu_{\mathbf x^{-i}}, y)\)
and denote by \(\hat \mu_{\mathbf x^{-i}}\)
the probability on \(\mathbb R^{2d}\) such that
\[
\hat \mu_{\mathbf x^{-i}} (dxdv) \propto
\exp \biggl( - \frac{\delta F}{\delta m} (\mu_{\mathbf x^{-i}}, x)
- \frac 12 |v|^2 \biggr) dxdv.
\]
For every \(\varepsilon \in (0,1)\), the Fisher information satisfies
\begin{align*}
I_t&
=\sum_{i=1}^N \Expect \biggl[
\Bigl| \nabla_{x^i} \log m^{N,i|-i}_t\bigl(Z^i_t\big|\mathbf Z_t^{-i}\bigr)
+ D_m F\bigl(\mu_{\mathbf X^{-i}_t}, X^i_t\bigr) + V^i_t
+ \delta^i_1\bigl(\mathbf X_t; X^i_t\bigr)\Bigr|^2 \biggr] \\
&\geq \sum_{i=1}^N \Expect\!\left[
\begin{multlined}
(1 - \varepsilon) \Bigl| \nabla_{x^i} \log m^{N,i|-i}_t\bigl(X^i_t\big|\mathbf X_t^{-i}\bigr)
+ D_m F\bigl(\mu_{\mathbf X_t^{-i}}, X^i_t\bigr) + V^i_t \Bigr|^2 \\
- (\varepsilon^{-1} -1) \bigl|\delta^i_1(\mathbf X_t; X^i_t)\bigr|^2
\end{multlined}\right]
\\
&= (1 - \varepsilon) \sum_{i=1}^N \Expect\Bigl[ I \Bigl(m^{N,i|-i}_t\bigl(\cdot\big|\mathbf Z_t^{-i}\bigr) \Big| \hat \mu_{\mathbf X_t^{-i}}\Bigr)\Bigr]
- (\varepsilon^{-1} - 1) \sum_{i=1}^N \Expect \Bigl[\bigl|\delta^i_1\bigl(\mathbf X_t; X^i_t\bigr)\bigr|^2\Bigr],
\end{align*}
where we used the elementary inequality \((a + b)^2 \geq (1 - \varepsilon) |a|^2 - (\varepsilon^{-1} - 1) |b|^2\).
Define the first error
\begin{equation}
\label{eq:Delta_1}
\Delta_1 \coloneqq \sum_{i=1}^N \Expect \Bigl[\bigl|\delta^i_1\bigl(\mathbf X_t; X^i_t\bigr)\bigr|^2\Bigr]
\coloneqq \sum_{i=1}^N \Expect \Bigl[\bigl|D_m F\bigl(\mu_{\mathbf X_t},X_t^i\bigr)
- D_m F\bigl(\mu_{\mathbf X_t^{-i}}, X_t^i\bigr)\bigr|^2\Bigr].
\end{equation}
The previous inequality writes
\begin{equation}
\label{eq:I-lower-bound-1}
I_t \geq
(1 - \varepsilon) \sum_{i=1}^N \Expect\Bigl[ I \Bigl(m^{N,i|-i}_t\bigl(\cdot\big|\mathbf Z_t^{-i}\bigr) \Big| \hat \mu_{\mathbf X_t^{-i}}\Bigr) \Bigr] - (\varepsilon^{-1} - 1) \Delta_1.
\end{equation}
We apply the uniform \(\rho\)-log-Sobolev inequality \cref{eq:lsi} for \(\hat \mu_{\mathbf X^i_t}\) with \(\rho\) defined by \cref{eq:defn-rho}
and obtain
\begin{multline*}
\frac{1}{4\rho}I\Bigl(m^{N,i|-i}_t\bigl(\cdot\big|\mathbf Z^{-i}_t\bigr) \Big | \hat \mu_{\mathbf X^{-i}_t}\Bigr)
\geq H\Bigl(m^{N,i|-i}_t\bigl(\cdot\big|\mathbf Z^{-i}_t\bigr) \Big| \hat \mu_{\mathbf X^{-i}_t}\Bigr) \\
= \int \biggl( \log m^{N,i|-i}_t\bigl(x^i \big| \mathbf Z^{-i}_t\bigr)
+ \frac{\delta F}{\delta m} \bigl(\mu_{\mathbf X^{-i}_t}, x^i\bigr)
+ \frac 12 |v^i|^2\biggr) m^{N,i|-i}_t\bigl(dz^i\big|\mathbf Z^{-i}_t\bigr) \\
+ \log Z\bigl(\hat \mu_{\mathbf X^{-i}_t}\bigr),
\end{multline*}
where the last quantity is the normalization factor
\[
Z\bigl(\hat \mu_{\mathbf X^{-i}_t} \bigr)
\coloneqq \int \exp \biggl( - \frac{\delta F}{\delta m}
\bigl( \mu_{\mathbf X^{-i}_t}, x \bigr)
- \frac 12 \lvert v \rvert^2\biggr) dxdv.
\]
Then we apply Jensen's inequality to \(\log Z (\hat\mu_{\mathbf x^{-i}})\) to obtain
\[
\log Z (\hat\mu_{\mathbf X^{-i}_t})
\geq - \int \biggl( \frac{\delta F}{\delta m}\bigl(\mu_{\mathbf X^{-i}_t}, x^i\bigr)
+ \frac 12|v^i|^2 \biggr) m_\infty(dz^i) - \int m_\infty(z^i) \log m_\infty(z^i) dz^i.
\]
Chaining the previous two inequalities and summing over \(i\),
we have
\begin{multline}
\label{eq:I-i|-i-lower-bound}
\frac{1}{4\rho} \sum_{i=1}^N
I\Bigl(m^{N,i|-i}_t\bigl(\cdot\big|\mathbf Z^{-i}_t\bigr)
\Big| \hat \mu_{\mathbf X^{-i}_t}\Bigr)
\geq \sum_{i=1}^N \biggl[ \int \biggr( \frac{\delta F}{\delta m} (\mu_{\mathbf X_t^{-i}}, x^i) + \frac 12 |v^i|^2 \biggl) \\
\Bigl(m^{N,i|-i}_t\bigl(dz^i\big|\mathbf Z_t^{-i}\bigr) - m_\infty(dz^i)\Bigr)
+ H\Bigl(m^{N,i|-i}_t\bigl(\cdot\big|\mathbf Z_t^{-i}\bigr)\Bigr) - H(m_\infty)
\biggr].
\end{multline}

\proofstep{Another change of empirical measure}
We are going to replace \(\mu_{\mathbf x^{-i}}\) by
\(\mu_{\mathbf x}\) in \cref{eq:I-i|-i-lower-bound}.
Define
\(
\delta^i_2 (\mathbf x; y) \coloneqq \frac{\delta F}{\delta m} (\mu_{\mathbf x^{-i}}, y)
- \frac{\delta F}{\delta m} (\mu_{\mathbf x}, y)
\)
and the second error
\begin{equation}
\Delta_2
\coloneqq \sum_{i=1}^N \int \delta^i_2 (\mathbf x; x^i) m^N_t(d\mathbf z)
- \sum_{i=1}^N \iint \delta^i_2 (\mathbf x; x') m_\infty (dz') m^N_t(d\mathbf z). \label{eq:Delta_2}
\end{equation}
Taking expectations on both sides of \cref{eq:I-i|-i-lower-bound}, we obtain
\begin{multline}
\frac{1}{4\rho} \sum_{i=1}^N \Expect
\Bigl[I\Bigl(m^{N,i|-i}_t\bigl(\cdot\big|\mathbf Z^{-i}_t\bigr) \Big| \hat \mu_{\mathbf X^{-i}_t}\Bigr)\Bigr] \\
\geq N \Expect \biggl[ \int \biggl( \frac{\delta F}{\delta m} (\mu_{\mathbf X_t}, x) + \frac12 |v|^2 \biggr) (\mu_{\mathbf Z_t} - m_{\infty}) (dz) \biggr]
+ \sum_{i=1}^N \Expect \Bigl[H\Bigl(m^{N,i|-i}_t\bigl(\cdot\big|\mathbf Z_t^{-i}\bigr)\Bigr)\Bigr] \\
- N H(m_\infty) + \Delta_2.
\end{multline}
Thanks to the convexity of \(F\), the first term satisfies the tangent inequality
\begin{multline}
\label{eq:ps-entropy-convergence-tangent}
N \Expect \biggl[ \int \biggl( \frac{\delta F}{\delta m} (\mu_{\mathbf X_t}, x) + \frac12 |v|^2 \biggr) (\mu_{\mathbf Z_t} - m_{\infty}) (dz) \biggr] \\
\geq N \Expect \bigl[ F\bigl(\mu_{\mathbf X_t}\bigr) - F\bigl(m^x_{\infty}\bigr) \bigr]
+ \frac 12 \int |\mathbf v|^2 m^N_t(d\mathbf z) - \frac N2 \int |v^2| m_\infty(dz) \\
= F^N\bigl(m^N_t\bigr) - NF(m_{\infty})
+ \frac 12 \int |\mathbf v|^2 m^N_t(d\mathbf z) - \frac N2 \int |v^2| m_\infty(dz).
\end{multline}
For the second term we apply the information inequality \cref{eq:info-ineqs} to obtain
\[
\sum_{i=1}^N \Expect^{-i} \Bigl[
H\Bigl(m^{N,i|-i}_t\bigl(\cdot\big|\mathbf Z_t^{-i}\bigr)\Bigr)\Bigr]
\geq H\bigl(m^N_t\bigr).
\]
Hence,
\[
\sum_{i=1}^N \Expect
\Bigl[I\Bigl(m^{N,i|-i}_t\bigl(\cdot\big|\mathbf Z^{-i}_t\bigr)
\Big| \hat \mu_{\mathbf X^{-i}_t}\Bigr)\Bigr]
\geq 4\rho
\Bigl(\mathcal F^N\bigl(m^N_t\bigr) - N\mathcal F(m_\infty) + \Delta_2\Bigr)
\]
by the definition of free energies
\(\mathcal F(m) = F(m) + H(m) + \frac 12 \int |v|^2 m\),
\(\mathcal F^N(m^N) = F^N(m^N) + H(m^N) + \frac 12 \int |\mathbf v|^2 m^N\).
Using \cref{eq:I-lower-bound-1},
we obtain
\begin{equation}
\label{eq:I-lower-bound-2}
I_t = I\bigl(m^N_t \big| m^N_\infty\bigr)
\geq 4\rho (1 - \varepsilon)
\Bigl(\mathcal F^N \bigl(m^N_t\bigr) - N\mathcal F(m_\infty) + \Delta_2\Bigr)
- (\varepsilon^{-1} - 1) \Delta_1.
\end{equation}

\proofstep{Bounding the errors \(\Delta_1\), \(\Delta_2\)}
The transport plan between \(\mu_{\mathbf x}\) and \(\mu_{\mathbf x^{-i}}\)
\begin{equation}
\label{eq:transport-plan-empiricals}
\pi^i = \frac 1N \sum_{j\neq i} \delta_{(x^j,x^j)}
+ \frac 1{N(N-1)} \sum_{j \neq i} \delta_{(x^j,x^i)}
\end{equation}
gives the bound
\(W_1 (\mu_{\mathbf x}, \mu_{\mathbf x^{-i}})
\leq \frac{1}{N(N-1)} \sum_{j\neq i}|x^j - x^i|\).
We use this transport plan to bound the errors \(\Delta_1\), \(\Delta_2\).

Let us treat the first error \(\Delta_1\).
Since \(m \mapsto D_m F(m,x)\) is \(M^F_{mm}\)-Lipschitz continuous in \(W_2\) metric,
we have
\[
\bigl|\delta^i_1(\mathbf x; y)\bigr|
\leq M^F_{mm} W_2 (\mu_{\mathbf x}, \mu_{\mathbf x^{-i}}).
\]
Under the \(L^2\)-optimal transport plan
\(\Law \bigl((Z^i_t)_{i=1}^N, (\tilde Z^i_\infty)_{i=1}^N\bigr) \in \Pi(m^N_t, m^{\otimes N}_\infty)\)
we have
\begin{align*}
\Delta_1 &= \sum_{i=1}^N \Expect\Bigl[ \bigl|\delta^i_1 (\mathbf X_t; X^i_t)\bigr|^2 \Bigr]
\leq\bigl(M^F_{mm}\bigr)^2 \sum_{i=1}^N \Expect\bigl[W_1^2\bigl(\mu_{\mathbf X_t}, \mu_{\mathbf X^{-i}_t}\bigr)\bigr] \\
&\leq \frac{\bigl(M^F_{mm}\bigr)^2}{N(N-1)} \Expect \biggl[ \sum_{\substack{1 \leq i,j \leq N \\ i\neq j}}\bigl|X^j_t - X^i_t\bigr|^2 \biggr] \\
&\leq \frac{3\bigl(M^F_{mm}\bigr)^2}{N(N-1)} \Expect \biggl[ \sum_{\substack{1 \leq i,j \leq N \\ i\neq j}}
\Bigl(\bigl|X^i_t - \tilde X^i_\infty\bigr|^2
+ \bigl|\tilde X^i_\infty - \tilde X^j_\infty\bigr|^2
+ \bigl|X^j_t - \tilde X^j_\infty\bigr|^2\Bigr) \biggr] \\
&\leq \frac{3\bigl(M^F_{mm}\bigr)^2}{N(N-1)}
\biggl(
2(N-1)\Expect \biggl[
\sum_{i=1}^N\bigl|X^i_t - \tilde X^i_\infty\bigr|^2 \biggr]
+ N(N-1) \Expect \Bigl[\bigl|\tilde X^1_\infty - \tilde X^2_\infty\bigr|^2\Bigr]
\biggr).
\end{align*}
The first term
\(\Expect \bigl[\sum_{i=1}^N|X^i_t - \tilde X^i_\infty|^2 \bigr]\)
is bounded by the Wasserstein distance \(W_2^2\bigl(m^N_t, m_\infty^{\otimes N}\bigr)\),
while the second \(\Expect \bigl[|\tilde X^1_\infty - \tilde X^2_\infty|^2\bigr]\)
equals \(2 \Var m^x_\infty\).
Hence the first error satisfies the bound
\begin{equation}
\Delta_1 \leq 6\bigl(M^F_{mm}\bigr)^2
\biggl( \frac 1N W_2^2\bigl(m^N_t, m^{\otimes N}_\infty\bigr)
+ \Var m_\infty \biggr). \label{eq:bound-Delta_1}
\end{equation}

Now treat the second error \(\Delta_2\).
The Lipschitz constant of
\(y \mapsto \delta^i_2(\mathbf x; y)
= \frac{\delta F}{\delta m}(\mu_{\mathbf x^{-i}}, y)
- \frac{\delta F}{\delta m}(\mu_{\mathbf x}, y)\)
is controlled by
\[
\bigl| \nabla_y \delta^i_2(\mathbf x; y) \bigr|
= \lvert D_m F(\mu_{\mathbf x}, y) - D_m F(\mu_{\mathbf x^{-i}}, y)\rvert
\leq M^F_{mm} W_1 (\mu_{\mathbf x}, \mu_{\mathbf x^{-i}}).
\]
Hence \(\bigl|\delta^i_2(\mathbf x; y) - \delta^i_2(\mathbf x; y')\bigr|
\leq M^F_{mm} W_1(\mu_{\mathbf x}, \mu_{\mathbf x^{-i}}) |y - y'|\).
Use Fubini's theorem to first integrate \(z'\) in the definition of the second error \cref{eq:Delta_2}
and let \(\tilde Z'_\infty\) be independent from \(\mathbf Z_t\).
We obtain
\begin{align*}
|\Delta_2| &\leq \sum_{i=1}^N \int \biggl( \int \bigl|\delta^i_2(\mathbf x; x^i) - \delta^i_2(\mathbf x; x')\bigr| m_\infty (dz') \biggr) m^N_t(d\mathbf z) \\
&\leq \sum_{i=1}^N \iint M^F_{mm}
W_1(\mu_{\mathbf x}, \mu_{\mathbf x^{-i}}) \lvert x' - x^i\rvert
m_\infty(dz') m^N_t(d\mathbf z) \\
&\leq \sum_{i=1}^N \iint \frac{M^F_{mm}}{N(N-1)} \sum_{j=1,\,j\neq i}^N
\lvert x^i - x^j \rvert |x' - x^i| m_\infty(dz') m^N_t(d\mathbf z) \\
&\leq \frac{M^F_{mm}}{2N(N-1)}
\sum_{i=1}^N \iint \sum_{j=1,\,j\neq i}^N
\bigl(\lvert x^i - x^j \rvert^2 + \lvert x' - x^i\rvert^2\bigr)
m_\infty(dz') m^N_t(d\mathbf z) \\
&\leq \frac{M^F_{mm}}{2N(N-1)}
\biggl(\sum_{\substack{i,j=1 \\ i\neq j}}^N
\Expect \Bigl[ \bigl| X^i_t - X^j_t\bigr|^2 \Bigr]
+ (N-1) \sum_{i=1}^N \Expect \Bigl[ \bigl|X^i_t - \tilde X'_\infty\bigr|^2 \Bigr]
\biggr).
\end{align*}
Using the same method we used for \(\Delta_1\),
we control the first term by
\[
\sum_{\substack{i,j=1 \\ i\neq j}}^N
\Expect \Bigl[ \bigl|X^i_t - X^j_t\bigr|^2 \Bigr]
\leq 6 N(N-1)\biggl(\frac 1N W_2^2\bigl(m_t^N, m_\infty^{\otimes N}\bigr)
+ \Var m^x_\infty\biggr).
\]
For the second term we work again under the \(L^2\)-optimal plan
\(\Law \bigl((Z^i_t)_{i=1}^N, (\tilde Z^i_\infty)_{i=1}^N\bigr) \in \Pi(m^N_t, m^{\otimes N}_\infty)\)
and let \(\tilde Z'_\infty\) remain independent from the other variables.
We have
\begin{multline*}
\sum_{i=1}^N \Expect \Bigl[ \bigl|X^i_t - \tilde X'_\infty\bigr|^2 \Bigr]
\leq 2 \sum_{i=1}^N \Bigl(\Expect \Bigl[\bigl|X^i_t - \tilde X^i_\infty\bigr|^2 \Bigr]
+ \Expect \Bigl[ \bigl|\tilde X^i_\infty - \tilde X'_\infty\bigr|^2 \Bigr] \Bigr) \\
\leq 2 N \biggl(\frac 1NW_2^2\bigl(m_t^N, m_\infty^{\otimes N}\bigr)
+ 2 \Var m^x_\infty\biggr).
\end{multline*}
As a result,
\begin{equation}
\label{eq:bound-Delta_2}
|\Delta_2| \leq M^F_{mm} \biggl( \frac 4N W_2^2\bigl(m_t^N,m_\infty^{\otimes N}\bigr)
+ 5\Var m^x_\infty \biggr).
\end{equation}

\proofstep{Conclusion}
Inserting the bounds on the errors \cref{eq:bound-Delta_1,eq:bound-Delta_2}
to the lower bound of Fisher information \cref{eq:I-lower-bound-2},
we obtain
\begin{align*}
I\bigl(m^N_t \big| m^N_\infty\bigr)
&\geq 4 \rho (1 - \varepsilon) \Bigl(\mathcal F^N\bigl(m^N_t\bigr) - N\mathcal F(m_\infty) \Bigr) \\
&\quad- \Bigl(16\rho M^F_{mm} + 6(\varepsilon^{-1} - 1) \bigl(M^F_{mm}\bigr)^2\Bigr) \frac 1N W_2^2\bigl(m_t^N, m^{\otimes N}_\infty\bigr) \\
&\quad- \Bigl(20\rho M^F_{mm} + 6(\varepsilon^{-1} - 1) \bigl(M^F_{mm}\bigr)^2\Bigr) \Var m^x_\infty.
\end{align*}
Thanks to the Poincaré inequality \cref{eq:poincare} for \(m_\infty = \hat m_\infty\),
its spatial variance satisfies
\begin{equation}
\label{eq:var-m-poincare}
2\rho \Var_{m_\infty} (x^i) \leq \Expect_{m_\infty}\bigl[|\nabla x^i|^2\bigr]= 1.
\end{equation}
So \(\Var m^x_\infty = \sum_{i=1}^d \Var_{m_\infty} (x^i) \leq d/2\rho\).
Using the \(T_2\)-transport inequality \cref{eq:t2} for \(m^{\otimes N}_\infty\) and the entropy sandwich \cref{lem:ps-entropy-sandwich}
we bound the transport cost by
\[
W_2^2\bigl(m^N_t, m^{\otimes N}_\infty\bigr)
\leq \frac 1{\rho} H\bigl(m^N_t\big| m^{\otimes N}_\infty\bigr)
\leq \frac 1{\rho} \Bigl(\mathcal F^N\bigl(m^N_t\bigr) - N \mathcal F(m_t)\Bigr).
\]
In the end we obtain \(\mathcal E^N\bigl(m^N_T\bigr)
\leq \mathcal E^N\bigl(m^N_s\bigr) - \alpha \int_s^T I_t dt\) where
\begin{multline*}
I_t = \frac{1}{2} I\bigl(m^N_t \big| m^N_\infty\bigr)
+ \frac{1}{2}I\bigl(m^N_t \big| m^N_\infty\bigr)\\
\geq
\frac 1 2 \biggl[ 4(1-\varepsilon)\rho - \frac {M^F_{mm}}N \biggl( 16 + 6(\varepsilon^{-1} - 1) \frac{M^F_{mm}}{\rho} \biggr) \biggr]
\Bigl( \mathcal F^N\bigl(m^N_t\bigr) - N\mathcal F(m_\infty) \Bigr) \\
+ \frac 12 I\bigl(m^N_t \big| m^N_\infty\bigr)
- \frac{dM^F_{mm}}{2\rho}\bigl(10 \rho + 3(\varepsilon^{-1} - 1) M^F_{mm} \bigr).
\end{multline*}
We conclude by applying Grönwall's lemma,
as in the end of the proof of \cref{thm:mf-entropy-convergence}.
\end{proof}

\section{Short-time behaviors and propagation of chaos}
\label{sec:short-time-poc}

Our proof of the main theorem on the uniform-in-time
propagation of chaos (\cref{thm:poc}) relies on the exponential convergence
in \cref{thm:mf-entropy-convergence,thm:ps-entropy-convergence},
where the initial conditions are required to have
finite entropy and finite Fisher information.
We aim to demonstrate in this section that
the non-linear kinetic Langevin dynamics exhibits
the same regularization effects in short time as the linear ones,
where the contributions from the non-linearity can be controlled.
We will first show the short-time Wasserstein propagation of chaos
using synchronous coupling.
Then we adapt the regularization results for the linear dynamics to our setting
and show that for measure initial values of finite second moment,
the entropy and the Fisher information are finite for the flow
at every positive time, where the short-time Wasserstein propagation of chaos
also plays a role.
Finally we combine all the estimates obtained to derive \cref{thm:poc}.

\subsection{Synchronous coupling}
\label{sec:synchronous-coupling}

We first show a lemma where synchronous coupling
is applied to general McKean--Vlasov diffusions.
This lemma is also used to justify the approximation arguments
in the proof of \cref{thm:mf-entropy-convergence,thm:ps-entropy-convergence}.

\begin{lem}
\label{lem:synchronous-coupling}
Let \(T > 0\) and
\(\beta\), \(\beta' : [0,T] \times \mathcal P_2(\mathbb R^{d}) \times \mathbb R^{d} \to \mathbb R^d\)
be measurable and uniformly Lipschitz continuous in the last two variables
and \(\sigma\) be a \(d \times d\) real matrix.
Suppose the integral \(\int_0^T (\lvert\beta(t,\delta_0,0)\rvert
+ \lvert\beta'(t,\delta_0,0)\rvert) dt\) is finite.
Let \((Z_t)_{t \in [0,T]}\), \((Z'_t)_{t \in [0,T]}\) be respective solutions to
\begin{align*}
dZ_t &= \beta\bigl(t, \Law(Z_t), Z_t\bigr) dt + \sigma dW_t, \\
dZ'_t &= \beta'\bigl(t, \Law(Z'_t), Z'_t\bigr) dt + \sigma dW'_t,
\end{align*}
where \(W\), \(W'\) are \(d\)-dimensional Brownians.
If there exist constants \(M_m\), \(M_z\)
and a progressively measurable \(\delta : \Omega \times [0,T] \to \mathbb R\)
such that for every \(t \in [0,T]\),
every \(m\), \(m' \in \mathcal P_2(\mathbb R^d)\)
and every \(x\), \(x' \in \mathbb R^d\),
\begin{equation}
\label{eq:synchronous-coupling-bound}
\lvert\beta(t,m,Z_t) - \beta'(t,m',Z'_t)\rvert
\leq M_m W_2(m, m') + M_z |Z_t - Z'_t| + \delta_t
\end{equation}
almost surely,
then for every \(t \in [0,T]\),
\begin{multline*}
W^2_2\bigl(\Law (Z_t), \Law (Z'_t)\bigr)
\leq e^{ (2M_m + 2M_z) t + 1 } W^2_2\bigl(\Law (Z_0), \Law (Z'_0)\bigr) \\
+ e^2t \int_0^t e^{(2M_m + 2M_z) (t - s)} \Expect\bigl[\delta_s^2\bigr] ds.
\end{multline*}
\end{lem}

\begin{proof}
From the uniformly Lipschitz continuity of \(b\) and \(b'\)
we have the uniqueness in law and the existence of strong solution
for both diffusions.
So we can construct \((Z_t,Z'_t)_{t \in [0,T]}\)
such that they share the same Brownian motion
and satisfy
\[
\Expect \bigl[ |Z_0 - Z'_0 |^2 \bigr]
= W_2^2 \bigl(\Law(Z_0), \Law(Z'_0) \bigr).
\]
Consequently,
\[
d (Z_s - Z'_s) =
\bigl[ b\bigl(s, \Law(Z_t), Z_s\bigr)
- b'\bigl(s, \Law(Z'_s), Z'_s\bigr) \bigr] dt
\]
and by Itō's formula,
\[
d |Z_s - Z'_s|^2 =
2 (Z_s - Z'_s) \cdot \bigl[ \beta\bigl(s, \Law(Z_s), Z_s\bigr)
- \beta'\bigl(s, \Law(Z'_s), Z'_s\bigr) \bigr] ds.
\]
By \cref{eq:synchronous-coupling-bound} we have
\begin{multline*}
\bigl| \beta\bigl(s, \Law(Z_s), Z_s\bigr) - \beta'\bigl(s, \Law(Z_s), Z_s\bigr) \bigr|
\leq M_m W_2 \bigl(\Law(Z_s), \Law(Z'_s)\bigr) \\ + M_z |Z_s - Z'_s| + \delta_s
\leq M_m \Expect \bigl[ |Z_s - Z'_s|^2 \bigr]^{1/2} + M_z |Z_s - Z'_s| + \delta_s.
\end{multline*}
Hence
\[
\frac 12d |Z_s - Z'_s|^2 \leq M_z |Z_s - Z'_s|^2 + M_m \Expect \bigl[ |Z_s - Z'_s|^2 \bigr]^{1/2} |Z_s - Z'_s| + |Z_s - Z'_s| \delta_s.
\]
By Cauchy--Schwarz,
\[
d |Z_s - Z'_s|^2
\leq (2M_z + M_m + t^{-1}) |Z_s - Z'_s|^2
+ M_m \Expect \bigl[ |Z_s - Z'_s|^2 \bigr]
+ t \delta_s^2.
\]
Taking expectations on both sides
and applying Grönwall's lemma,
we obtain
\begin{multline*}
\Expect \bigl[ |Z_t - Z'_t|^2 \big]
= e^{2(M_z + M_m)t + 1} \Expect \bigl[ |Z_0 - Z'_0|^2 \big]
+ t \int_0^t e^{2(M_z + M_m + t^{-1}) (t - s)} \Expect\bigl[\delta_s^2\bigr] ds \\
\leq e^{2(M_z + M_m)t + 1} \Expect \bigl[ |Z_0 - Z'_0|^2 \big]
+ e^2t \int_0^t e^{2(M_z + M_m) (t - s)} \Expect\bigl[\delta_s^2\bigr] ds,
\end{multline*}
from which the desired inequality follows.
\end{proof}

Since the finite-time propagation of chaos does not depend on the
gradient structure of the diffusions,
we introduce a more general setting.
Let \(b : \mathcal P_2(\mathbb R^{2d}) \times \mathbb R^d \times \mathbb R^d \to \mathbb R\) be a mapping that is Lipschitz in space and velocity:
there exist positive constants \(M^b_x\), \(M^b_v\) such that
\begin{multline}
\label{eq:b-lip}
\forall x,x',v,v' \in \mathbb R^d,~\forall m \in \mathcal P_2(\mathbb R^{2d}), \\
\lvert b(m,x,v) - b(m,x',v')\rvert
\leq M^b_x |x - x'| + M^b_v |v - v'|.
\end{multline}
We suppose also that the functional derivatives
\(\frac{\delta b}{\delta m}, \frac{\delta^2b}{\delta m^2}\) exist
with the following bounds:
there exist positive constants \(M^b_m\), \(M^b_{mm}\) such that
\begin{equation}
\label{eq:b-first}
\forall m \in \mathcal P_2(\mathbb R^{2d}),~\forall z, z' \in \mathbb R^{2d},\quad
\lvert D_m b(m,z, z')\rvert_\textnormal{op} \leq M^b_{m},
\end{equation}
and
\begin{multline}
\label{eq:b-second}
\forall m,m' \in \mathcal P_2(\mathbb R^{2d}),~\forall z \in \mathbb R^{2d}, \\
\biggl| \iint
\biggl[ \frac{\delta^2 b}{\delta m^2} (m',z, z', z') -
\frac{\delta^2 b}{\delta m^2} (m',z, z', z'') \biggr] m(dz') m(dz'')
\biggr|
\leq M^b_{mm}.
\end{multline}
We consider the following mean field dynamics:
\begin{equation}
\label{eq:b-mf-sde}
\begin{aligned}
dX_t &= V_t dt, \\
dV_t &= b\bigl( \Law(X_t,V_t),X_t,V_t\bigr) dt + \sqrt 2 dW_t,
\end{aligned}
\end{equation}
and the corresponding particle system:
\begin{equation}
\label{eq:b-ps-sde}
\begin{aligned}
dX^i_t &= V^i_t dt, \\
dV^i_t &= b\bigl( \mu_{(\mathbf X_t, \mathbf V_t)}, X^i_t,V^i_t\bigr) dt
+ \sqrt 2 dW^i_t,
\quad\text{where $\mu_{(\mathbf X_t, \mathbf V_t)} = \frac 1N \sum_{i=1}^N \delta _{(X^i_t, V^i_t)}$,}
\end{aligned}
\end{equation}
and \(i = 1\), \(\ldots\,\), \(N\).
In both equations \(W_t\), \(W^i_t\) are standard Brownians and \((W^i_t)_{i=1}^N\) are independent from each other.
The dynamics \cref{eq:b-mf-sde,eq:b-ps-sde} are well defined globally in time
thanks to the Lipschitz continuity \cref{eq:b-lip}
and we denote by \(P\) and \(P^N\) the respective associated semigroups.
That is to say,
if \((X_t,V_t)\) solves \cref{eq:b-mf-sde} and \(\Law(X_0,V_0) = \mu\),
then \(P^*_t \mu \coloneqq \Law(X_t, V_t)\) and \((P_t f) (\mu) \coloneqq \bigl\langle f, (P_t)^* \mu \bigr\rangle\) for bounded measurable \(f : \mathbb R^{2d} \to \mathbb R\);
if \((X^i_t,V^i_t)_{i=1}^N\) solves \cref{eq:b-ps-sde} and \(\Law(\mathbf X_0,\mathbf V_0) = \mu^N\),
then \(\bigl(P^N_t\bigr)^{\!*} \mu^N \coloneqq \Law(\mathbf X_t, \mathbf V_t)\)
and \(\bigl(P^N_t f^N\bigr) (\mu^N) \coloneqq
\bigl\langle f^N, \bigl(P^N_t\bigr)^{\!*} \mu^N \bigr\rangle\) for bounded measurable \(f^N : \mathbb R^{2dN} \to \mathbb R\).
We also define the tensor product of the mean field semigroup:
\(\bigl(P^{\otimes N}_t f^N\bigr) (\mu)
\coloneqq \bigl\langle f^N, (P^*_t \mu)^{\otimes N}\bigr\rangle\).

Using the previous \cref{lem:synchronous-coupling} as a building block,
we now show the finite-time propagation of chaos result.

\begin{prop}[Finite-time propagation of chaos]
\label{prop:finite-w-poc}
Assume \(b\) satisfies \cref{eq:b-lip,eq:b-first,eq:b-second}
and let \(N \geq 2\).
Then there exist a positive constant \(C\)
depending on \(M^b_x\), \(M^b_v\), \(M^b_m\) and \(M^b_{mm}\) such that
for every \(m \in \mathcal P_2(\mathbb R^{2d})\) and \(m^N \in \mathcal P_2(\mathbb R^{2dN})\),
and every \(T \geq 0\),
\begin{equation}
\label{eq:finite-w-poc}
W_2^2 \Bigl( \bigl(P^N_T\bigr)^{\!*} m^N, (P_T^* m)^{\otimes N} \Bigr)
\leq C e^{CT} W_2^2\bigl(m^N, m^{\otimes N}\bigr) + C(e^{CT} - 1) (\Var m_0 + d).
\end{equation}
\end{prop}

\begin{proof}
We apply \cref{lem:synchronous-coupling} with
\begin{align*}
\beta^i (t, m^N, \mathbf z) &= \bigl( v^i, b(\mu_{(\mathbf x, \mathbf v)}, x^i, v^i) \bigr)^\mathsf{T}, \\
\beta'^i (t, m^N, \mathbf z) &= \bigl( v^i, b(P^*_tm, x^i, v^i) \bigr)^\mathsf{T}, \\
\delta^2_t &\coloneqq \sum_{i=1}^N \bigl|\delta^i_t\bigr|^2 \coloneqq
\sum_{i=1}^N \bigl| b\bigl(\mu_{(\mathbf X_t, \mathbf V_t)}, X^i_t, V^i_t\bigr)
- b\bigl(P^*_t m, X^i_t, V^i_t\bigr) \bigr|^2,
\end{align*}
and \(M_z \coloneqq \sqrt 2 M^b_x \vee \sqrt{ 2 (M^b_v)^2 + 1}\), \(M_m \coloneqq 0\).
We then obtain
\[
W_2^2\Bigl( \bigl(P^N_t\bigr)^{\!*} m^N, (P_t^* m)^{\otimes N} \Bigr)
\leq e^{2M_z t + 1} W_2^2\bigl(m^N_0, m_0^{\otimes N}\bigr)
+ e^2t \int_0^t e^{2M_z (t - s)} \Expect\bigl[ \delta_t^2\bigr] ds.
\]
So it remains to bound \(\Expect[\delta_t^2]\).
By enlarging the underlying probability space,
we construct the random variable
\(\tilde{\mathbf Z}_t = \bigl(\tilde{\mathbf X}'_t, \tilde{\mathbf V}'_t\bigr)
\sim (P_t^* m)^{\otimes N} \)
such that
\[
\sum_{i=1}^N \Expect \Bigl[\bigl|Z^i_t - {\tilde Z}'^{i}_t\bigr|^2\Bigr]
= W_2^2 \Bigl(\bigl(P^N_t\bigr)^{\!*} m^N, (P_t^* m)^{\otimes N} \Bigr).
\]
This implies in particular
\begin{equation}
\label{eq:bound-between-empirical}
\Expect\bigl[W_2^2\bigl(\mu_{(\mathbf X_t, \mathbf V_t)}, \mu_{(\tilde{\mathbf X}'_t, \tilde{\mathbf V}'_t)}\bigr) \bigr]
\leq \frac 1N W_2^2\bigl(m^N_t, m_t^{\otimes N}\bigr).
\end{equation}
For each \(i\), we decompose
\begin{multline*}
\delta^i_t
= \Bigl( b \bigl( P_t^* m, Z^i_t\bigr)
- b\bigl( \mu_{\tilde{\mathbf Z}'^{-i}_t}, Z^i_t\bigr)\Bigr)
+ \Bigl( b\bigl( \mu_{\tilde{\mathbf Z}'^{-i}_t}, Z^i_t\bigr)
- b\bigl( \mu_{\tilde{\mathbf Z}'_t}, Z^i_t\bigr)\Bigr) \\
+ \Bigl( b\bigl( \mu_{\tilde{\mathbf Z}'_t}, Z^i_t\bigr)
- b\bigl( \mu_{\mathbf Z_t}, Z^i_t\bigr)\Bigr)
\eqqcolon \text{(I)} + \text{(II)} + \text{(III)}.
\end{multline*}
According to the assumption \cref{eq:measure-third}
we can apply \cref{lem:convergence-empirical} to the first term and obtain
\[
\Expect \bigl[\text{(I)}^2\bigr]
= \Expect\Bigl[ \Expect\bigl[ \text{(I)}^2 \big| Z^i_t \bigr] \Bigr]
\leq \frac{\bigl(M^b_{m}\bigr)^2 \Var m_t}{N - 1} + \frac{M^b_{mm}}{4(N - 1)^2}.
\]
We then bound the second term by the \(M^b_{m}\)-Lipschitz continuity:
\begin{multline*}
\Expect\bigl[ \text{(II)}^2 \bigr] \leq
\bigl(M^b_{m}\bigr)^2 \Expect\bigl[ W_2^2\bigl(\mu_{\tilde{\mathbf Z}'^{-i}_t},
\mu_{\tilde{\mathbf Z}'_t}\bigr) \bigr] \\
\leq \frac{\bigl(M^b_{m}\bigr)^2}{N(N-1)}
\sum_{j : j \neq i} \Expect \Bigl[\bigl|\tilde Z'^{j}_t - \tilde Z'^{i}_t\bigr|^2\Bigr]
= \frac{2\bigl(M^b_{m}\bigr)^2}{N} \Var P_t^*m.
\end{multline*}
Finally by \cref{eq:bound-between-empirical},
we have
\[
\Expect\bigl[\text{(III)}^2 \bigr]
\leq (M^b_{m})^2 \Expect \bigl[ W_2^2 \big(\mu_{\mathbf Z_t}, \mu_{\tilde{\mathbf Z}'_t}\bigr) \bigr]
\leq \frac{\bigl(M^b_{m}\bigr)^2}{N}
W_2^2\Bigl(\bigl(P^N_t\bigr)^{\!*} m^N, (P_t^* m)^{\otimes N}\Bigr).
\]
Hence
\[
\Expect\bigl[\delta^2_t\bigr]
= \sum_{i=1}^N \Expect \bigl[ |\delta^i_t|^2 \bigr]
\leq C \Bigl[ 1 + \Var m_t
+ W_2^2\Bigl(\bigl(P^N_t\bigr)^{\!*} m^N, (P_t^* m)^{\otimes N}\Bigr) \Bigr]
\]
for some constant \(C = C\bigl(M^b_{m}, M^b_{mm}\bigr)\).
By Itō's formula the variance \(\Var m_t\) satisfies
\begin{multline*}
\frac{d}{dt} \Var m_t
= 2 \bigl(\Expect\bigl[\tilde V_t^2\bigr]
- \Expect[\tilde V_t]^2\bigr)
+ 2\Expect\bigr[\tilde X_t \cdot b(P_t^* m ,\tilde Z_t)\bigl] \\
- 2\Expect[\tilde X_t] \cdot \Expect\bigl[b(P_t^* m,\tilde Z_t)\bigr] + 2d
\leq C' (\Var m_t + d)
\end{multline*}
for some \(C' = C'\bigl(M^b_{z}, M^b_{m}\bigr)\).
Then Grönwall's lemma yields
\(\Var m_t \leq e^{C't} \Var m_0 + (e^{C't} - 1)d\).
Upon redefining the constants, we obtain for every \(t \geq 0\),
\begin{multline*}
W_2^2\Bigl(\bigl(P^N_t\bigr)^{\!*} m^N, (P_t^* m)^{\otimes N}\Bigr)
\leq e^{2M_{z} t + 1} W_2^2\bigl(m^N, m^{\otimes N}\bigr) \\
+ C t \int_0^t e^{2M_{z} (t - s)}
\Bigl[
W_2^2\Bigl(\bigl(P^N_s\bigr)^{\!*} m^N, (P_s^* m)^{\otimes N}\Bigr)
+ e^{Cs} (\Var m_0 + d) \Bigr] ds.
\end{multline*}
We then conclude by applying the integral version of Grönwall's lemma.
\end{proof}

\subsection{From Wasserstein metric to entropy}
\label{sec:log-harnack}

We study in this section a logarithmic Harnack's inequality for
kinetic McKean--Vlasov dynamics and the corresponding particle system.
This inequality then implies the regularization from Wasserstein to entropy.

\begin{lem}[Log-Harnack inequality for propagation of chaos]
\label{lem:log-harnack-poc}
Assume \(b\) satisfies \cref{eq:b-lip,eq:b-first,eq:b-second}
and let \(N \geq 2\).
Then there exist a positive constant \(C\)
depending on \(M^b_x\), \(M^b_v\), \(M^b_m\) and \(M^b_{mm}\) such that
for every \(m \in \mathcal P_2(\mathbb R^{2d})\) and \(m^N \in \mathcal P_2(\mathbb R^{2dN})\),
every measurable function \(f^N : \mathbb R^{2dN} \to (0, +\infty)\) that is lower bounded away from \(0\) and upper bounded,
and every \(T > 0\),
\begin{multline}
\label{eq:log-harnack-poc}
\bigl(P^N_T \log f^N\bigr) (m^N)
\leq \log \bigl(P_T^{\otimes N} f^N\bigr) (m)
+ C \biggl( \frac{1}{(T \wedge 1)^3} + e^{CT} \biggr) W_2^2\bigl(m^N, m^{\otimes N}\bigr)\\
+ C(e^{CT} - 1)(\Var m + d).
\end{multline}
Consequently,
\begin{multline}
\label{eq:reverse-t2}
H \Bigl( \bigl(P^N_T\bigr)^{\!*} m^N \Big| (P^*_T m)^{\otimes N} \Bigr)
\leq C \biggl( \frac{1}{(T \wedge 1)^3} + e^{CT} \biggr)
W_2^2\bigl(m^N, m^{\otimes N}\bigr) \\
+ C(e^{CT} - 1)(\Var m + d).
\end{multline}
\end{lem}

\begin{proof}
Let us first prove the log-Harnack inequality \cref{eq:log-harnack-poc}
for compactly supported \(m\) and \(m^N\).

\proofstep{Constructing a bridge}
Fix \(T > 0\) and let \((\tilde X^i_t, \tilde V^i_t)_{i=1}^N\) be \(N\) independent duplicates of the solution to \cref{eq:b-mf-sde}
with the initial condition \(\Law(\tilde X^i_0, \tilde V^i_0) = m\)
for \(i = 1\), \(\ldots\,\), \(N\).
We denote the \(N\)-independent Brownians by \(\tilde W^i_t\).
By enlarging the underlying probability space, we construct random variables \(\mathbf X_0, \mathbf V_0\) such that
\[
\sum_{i=1}^N \Expect \Bigl[ \bigl|X^i_0 - \tilde X^i_0\bigr|^2
+ \bigl|V^i_0 - \tilde V^i_0\bigr|^2 \Bigr]
= W_2^2\bigl(m^N, m^{\otimes N}\bigr).
\]
Define for \(i = 1\), \(\ldots\,\), \(N\) the stochastic processes
\begin{align}
dX^i_t &= V^i_t dt, \\
dV^i_t &= \biggl( b\bigl( P^*_t m, \tilde X^i_t, \tilde V^i_t\bigr)
- \frac{V^i_0 - \tilde V^i_0}{T}
+ \frac{d}{dt} \bigl(t(T-t)\bigr) v^i \biggr) dt + \sqrt 2 d\tilde W^i_t,
\end{align}
where
\begin{equation}
v^i \coloneqq \frac{6}{T^3} \biggl( - \bigl(X^i_0 - \tilde X^i_0\bigr)
+ \frac{T}{2} \bigl(V^i_0 - \tilde V^i_0\bigr) \biggr).
\end{equation}
The difference processes \(\bigl(X^i_t - \tilde X^i_t, V^i_t - \tilde V^i_t\bigr)\)
satisfy
\begin{align*}
d\bigl(V^i_t - \tilde V^i_t\bigr)
&= - \frac{V^i_0 - \tilde V^i_0}{T} dt + d\bigl(t(T-t)\bigr) v^i, \\
d\bigl(X^i_t - \tilde X^i_t\bigr)
&= \frac{t - T}{T}\bigl(V^i_0 - \tilde V^i_0\bigr) dt + t(T-t) v^i dt,
\end{align*}
so that
\begin{align}
V^i_t - \tilde V^i_t &= \frac{T-t}{T}\bigl(V^i_0 - \tilde V^i_0\bigr)
+ \frac{6t(T-t)}{T^3}
\biggl( -\bigl(X^i_0 - \tilde X^i_0\bigr)
+ \frac{T}{2} \bigl(V^i_0 - \tilde V^i_0\bigr) \biggr), \\
X^i_t - \tilde X^i_t &= - \frac{t(t-T)^2}{T}\bigl(V^i_0 - \tilde V^i_0\bigr)
+ \frac{T^3 - 3 T t^2 + 2t^3}{T^3}\bigl(X^i_0 - \tilde X^i_0\bigr).
\end{align}
In particular \(X^i_T = \tilde X^i_T\) and \(V^i_T = \tilde V^i_T\).

\proofstep{Change of measure}
Define
\[
\xi^i_t \coloneqq \frac 1{\sqrt 2} \biggl(
b\bigl(P^*_t m,\tilde X^i_t, \tilde V^i_t\bigr)
- b\bigl(\mu_{(\mathbf X_t, \mathbf V_t)}, X^i_t, V^i_t\bigr)
+ \frac{\tilde V^i_0 - V^i_0}{T} + \frac{d}{dt}\bigl(t(T-t)\bigr) v^i \biggr)
\]
and \(\delta b^i_t \coloneqq
b\bigl( P^*_t m, Z^i_t \bigr) - b\bigl( \mu_{\mathbf Z_t}, Z^i_t\bigr)\).
It satisfies
\begin{equation}
\label{eq:xi-bound}
\bigl|\xi^i_t\bigr|
\leq C \bigl|\delta b^i_t\bigr|
+ C \biggl( M^b_x + \frac{M^b_v}{T} + \frac 1{T^2} \biggr)
\bigl(\bigl|X^i_0 - \tilde X^i_0\bigr|
+ T \bigl|V^i_0 - \tilde V^i_0\bigr|\bigr)
\end{equation}
for some universal constant \(C\).
In the following \(C\) may change from line to line and
depend on the constants \(M^b_x\), \(M^b_v\), \(M^b_m\) and \(M^b_{mm}\).
Set \(W^i_\cdot \coloneqq \tilde W^i_\cdot + \int_0^\cdot \xi^i_t dt\) and

\[
R_\cdot \coloneqq \exp \biggl[ - \sum_{i=1}^N
\biggl( \int_0^\cdot \xi^i_t d\tilde W^i_t
+ \frac 12 \int_0^\cdot \bigl|\xi^i_s\bigr|^2 dt \biggr) \biggr]
\]
which is a local martingale.
Then \((X^i, V^i, W^i)\) solves \cref{eq:b-ps-sde}.
Since \(m\), \(m^N\) are both compactly supported,
\(\bigl|X^i_0 - \tilde X^i_0\bigr|\), \(\bigl|V^i_0 - \tilde V^i_0\bigr|\)
are bounded almost surely.
The difference in drift \(\delta b^i_t\)
has uniform linear growth in \(\mathbf X_t\), \(\mathbf V_t\),
and therefore uniform linear growth in \(\tilde{\mathbf X}_t\), \(\tilde{\mathbf V}_t\).
We then apply \cref{lem:girsanov-ui} in the appendix
to obtain that \(R_\cdot\) is really a martingale.
By Girsanov's theorem \(W^i_t\) are independent Brownians under the new probability \(\mathbb Q = R \mathbb P\).
Since \(\mathbf X_0\), \(\mathbf V_0\), \(\tilde{\mathbf X}_0\), \(\tilde{\mathbf V}_0\)
are independent from the Brownian motions we have
\[
\sum_{i=1}^N \Expect^{\mathbb Q}
\Bigl[ \bigl|X^i_0 - \tilde X^i_0\bigr|^2 + \bigl|V^i_0 - \tilde V^i_0\bigr|^2 \Bigr]
= W_2^2\bigl(m^N, m^{\otimes N}\bigr).
\]
Hence for measurable functions \(f^N : \mathbb R^{2dN} \to \mathbb R\) that are lower bounded away from \(0\) and upper bounded,
we have
\begin{align*}
\bigl(P^N_T \log f^N\bigr) (m^N)
&= \Expect \bigl[ R_T \log f^N (\mathbf X_T, \mathbf V_T) \bigr] \\
&\leq \Expect [R_T \log R_T] + \log \Expect \bigl[ f^N(\mathbf X_T, \mathbf V_T) \bigr] \\
&= \Expect [R_T \log R_T] + \log \Expect \bigl[ f^N(\tilde{\mathbf X}_T, \tilde{\mathbf V}_T) \bigr] \\
&= \Expect [R_T \log R_T] + \log\bigl(P_T^{\otimes N} f^N\bigr) (m).
\end{align*}
So it remains to bound \(\Expect [R_T \log R_T]\).
We observe
\begin{multline*}
\Expect [R_T \log R_T] = \Expect^{\mathbb Q} [\log R_T]
= \frac 12 \Expect^{\mathbb Q} \biggl[ \sum_{i=1}^N \int_0^T \bigl|\xi^i_t\bigr|^2 dt \biggr] \\
\leq CT \biggl( M^b_x + \frac{M^b_v}{T} + \frac 1{T^2} \biggr)^{\!2}
\Expect^{\mathbb Q} \Bigl[\bigl|X^i_0 - \tilde X^i_0\bigr|^2
+ T^2\bigl|V^i_0 - \tilde V^i_0\bigr|^2\Bigr]
+ C \Expect^{\mathbb Q} \biggl[ \int_0^T \sum_{i=1}^N\bigl|\delta b^i_t\bigr|^2 dt \biggr] \\
\leq \frac{C(T \vee 1)^3}{(T \wedge 1)^3} W_2^2\bigl(m^N, m^{\otimes N}\bigr)
+ C \Expect^{\mathbb Q} \biggl[ \int_0^T \sum_{i=1}^N \bigl|\delta b^i_t\bigr|^2 dt \biggr].
\end{multline*}
Arguing as in the proof of \cref{prop:finite-w-poc}, we have
\[
\sum_{i=1}^N \Expect^{\mathbb Q} \Bigl[\bigl|\delta b^i_t\bigr|^2\Bigr]
\leq C e^{Ct} \Bigl(W_2^2\bigl(m^N, m^{\otimes N}\bigr) + \Var m + d \Bigr),
\]
So the log-Harnack inequality \cref{eq:log-harnack-poc} is proved
for compactly supported \(m^N\) and \(m\).

\proofstep{Approximation}
Now treat general \(m^N\), \(m\) of finite second moment, but not necessarily compact supported.
Take two sequences \((m^N_k)_{k \in \mathbb N}, (m_k)_{k \in \mathbb N}\)
of compactly supported measures
such that \(m^N_k \to m^N\) and \(m_k \to m\) in respective topologies of \(\mathcal P_2\).
For continuous \(f^N\) such that \(\log f^N\) is bounded,
we have
\[
\bigl(P^N_t \log f^N\bigr) \bigl(m^N_k\bigr)
\to \bigl(P^N_t \log f^N\bigr) (m^N),
\qquad\bigl(P^{\otimes N}_t f^N\bigr) (m_k) \to\bigl(P^{\otimes N}_t f^N\bigr) (m)
\]
by the \(\mathcal P_2\)-continuities of \(\bigl(P^N_t\bigr)^{\!*}\) and \(P_t^*\).
So the log-Harnack inequality \cref{eq:log-harnack-poc} is shown for every continuous \(f^N\) which is both lower and upper bounded,
and for general \(m^N\) and \(m\) of finite second moment.
For a doubly bounded but not necessarily continuous \(f^N\) we take a sequence of continuous and uniformly bounded \((f^N_k)_{k \in \mathbb N}\) that converges to \(f^N\) in the \(\sigma(L^\infty, L^1)\) topology.
We have
\[
\bigl(P^N_t \log f^N_k\bigr) (m^N) \to \bigl(P^N_t \log f^N\bigr) (m^N),
\qquad \bigl(P^{\otimes N}_t f^N_k\bigr) (m) \to \bigl(P^{\otimes N}_t f^N\bigr) (m)
\]
since both \((P^N_t)^* m^N\) and \(P^*_t m\) are absolutely continuous with respect to the Lebesgue measure according to \cref{lem:wellposedness-regularity}.
So the desired inequality \cref{eq:log-harnack-poc} is shown in full generality.
Finally, to obtain \cref{eq:reverse-t2} we define another sequence
\[
g^N_k \coloneqq
\biggl(
\frac{\bigl(P^N_T\bigr)^{\!*} m^N}{(P^*_T m)^{\otimes N}}
\wedge k\biggr) \vee \frac 1k
\]
for \(k \in \mathbb N\).
We apply the Harnack's inequality \cref{eq:log-harnack-poc} to \(g^N_k\) and take the limit \(k \to +\infty\).
\end{proof}

Using the known results on log-Harnack inequalities
we can also obtain the regularization in the beginning of the dynamics.

\begin{prop}
\label{prop:w-to-h}
Assume \(F\) satisfies \cref{eq:lip}
and there exist probabilities \(m_\infty\), \(m^N_\infty\) satisfying \cref{eq:mf-invariant-measure,eq:ps-invariant-measure} respectively
and having finite exponential moments.
Let \(m_0\) (resp.\ \(m^N_0\)) be the initial value of
the mean field dynamics \cref{eq:mf-fp}
(resp.\ the particle system dynamics \cref{eq:ps-fp})
of finite second moment.
Then there exist a positive constant \(C\) depending
on \(M^F_{mm}\) and \(M^F_{mx}\) such that
for every \(t \in (0,1]\),
\begin{equation}
\label{eq:w-to-h}
H(m_t | m_\infty) \leq \frac{C}{t^3} W_2^2(m_0,m_\infty)
\quad\text{(resp.\ $H\bigl(m^N_t\big| m^N_\infty\bigr)
\leq \frac{C}{t^3} W_2^2\bigl(m^N_0, m^N_\infty\bigr)$).}
\end{equation}
\end{prop}

\begin{proof}
Note that \(m_t = P_t^* m_0\) and \(m^N_t = \bigl(P^N_t\bigr)^{\!*} m^N_0\)
where \(P_t\) and \(P^N_t\) are
the McKean--Vlasov and the linear semigroup
corresponding to the SDEs \cref{eq:mf-sde,eq:ps-sde}, respectively.
We then apply the log-Harnack inequality for McKean--Vlasov diffusions \cite[Proposition 5.1]{ren2021exponential}
and obtain
\[
H(m_t | m_\infty) \leq \frac{C}{t^3} W_2^2(m_t, m_\infty)
\]
for \(t \in (0,1]\).
For the particle system we apply the classical log-Harnack inequality
(which corresponds to the case where \(M^b_m\) and \(M^b_{mm}\) are both equal
to \(0\) in our \cref{lem:log-harnack-poc}, i.e.\ no mean field dependence)
and obtain
\[
H\bigl(m^N_t \big| m^N_\infty\bigr) \leq \frac{C}{t^3}
W_2^2\bigl(m^N_t, m^N_\infty\bigr)
\]
for \(t \in (0,1]\)
and it is clear from the computations in \cref{lem:log-harnack-poc} that
the constant \(C\) can be chosen to depend only on
\(M^F_{mx}\) and \(M^F_{mm}\).
\end{proof}

\subsection{From entropy to Fisher information}
\label{sec:herau}

We then adapt Hérau's functional to our setting to obtain
the regularization from entropy to Fisher information.

\begin{prop}
\label{prop:h-to-i}
Assume that \(F\) satisfies \cref{eq:lip,eq:convex},
and that there exist probabilities \(m_\infty\), \(m^N_\infty\) satisfying \cref{eq:mf-invariant-measure,eq:ps-invariant-measure} respectively
and having finite exponential moments.
Let \(m_0\) (resp.\ \(m^N_0\)) be the initial value of
the mean field dynamics \cref{eq:mf-fp}
(resp.\ the particle system dynamics \cref{eq:ps-fp})
of finite second moment and finite entropy.
Then there exist a positive constant \(C\) depending
on \(M^F_{mm}\) and \(M^F_{mx}\) such that
for every \(t \in (0,1]\),
\begin{equation}
\label{eq:h-to-i}
I(m_t | \hat m_t) \leq \frac{C}{t^3} \bigl(\mathcal F(m_0) - \mathcal F(m_\infty) \bigr)
\quad\text{(resp.\ $I\bigl(m^N_t \big| m^N_\infty\bigr)
\leq \frac{C}{t^3} H\bigl(m^N_0 \big| m^N_\infty\bigr)$).}
\end{equation}
\end{prop}

\begin{proof}
First derive the bound for the mean field system.
We suppose additionally \(F\) satisfies \cref{eq:space-higher}
and \(m_0 / m_\infty \in \mathcal A_+\) without loss of generality,
as they can be removed by the approximation argument
in the end of the proof of \cref{thm:mf-entropy-convergence}.
Let \(a\), \(b\), \(c\) be positive constants to be determined.
Motivated by \cite[Theorem A.18]{hypocoer},
we define Hérau's Lyapunov functional for mean field measures:
\begin{multline*}
\mathcal E (t,m) = F(m) + \frac 12\int |v^2| m + H(m)
+ a t \int |\nabla_v \log \eta|^2 m \\
+ 2b t^2 \int \nabla_x \log \eta\cdot \nabla_v \log \eta m
+ c t^3 \int |\nabla_x \log \eta|^2 m
\end{multline*}
where \(\eta \coloneqq m / \hat m\).
From the argument of \cref{thm:mf-entropy-convergence},
we know that \(\mathcal E(t,m_t)\) is well defined
and \(t \mapsto \mathcal E(t,m_t)\) admits derivative satisfying
\(\frac{d}{dt} \mathcal E(t,m_t) \leq - Y_t^{\mathsf T} K'_t Y_t\),
where \(K'_t\) is equal to
\[
\begin{pmatrix}
1 - a + 2at - 2\bigl(M^F_{mx} + M^F_{mm} \bigr) b t^2 & - 2bt^2
& - 2at - 4bt - 2M^F_{mm} ct^3 & 0 \\
0 & 2at & - 2M^F_{mx}ct^3 & -4bt^2 \\
0 & 0 & 2bt^2 - 3ct^2 & 0 \\
0 & 0 & 0 & 2ct^3
\end{pmatrix}
\]
and \(Y_t\) is defined by \cref{eq:Y}.
We then choose the constants \(a\), \(b\), \(c\)
depending only on \(M^F_{mx}\) and \(M^F_{mm}\)
such that \(ac > b^2\) and
\(K'_t \succeq 0\) for \(t \in [0,1]\).
Hence \(t \mapsto \mathcal E(t,m_t)\) is non-increasing on \([0,1]\)
and the Fisher bound follows: for every \(t \in (0,1]\),
\begin{align*}
I(m_t | \hat m_t)
&\leq \frac{C}{t^3} \bigl( \mathcal E(t,m_t) - \mathcal F(m_t) \bigr) \\
&\leq \frac{C}{t^3} \bigl( \mathcal E(t,m_t) - \mathcal F(m_\infty) \bigr) \\
&\leq \frac{C}{t^3} \bigl( \mathcal E(0,m_0) - \mathcal F(m_\infty) \bigr) \\
&= \frac{C}{t^3} \bigl( \mathcal F(m_0) - \mathcal F(m_\infty) \bigr).
\end{align*}
Here, in the second inequality,
we use $\mathcal F(m_t) \geq \mathcal F(m_\infty)$
which is a consequence of \cref{lem:mf-entropy-sandwich}.
Note that this inequality relies on the convexity of $F$.

For the particle system
we suppose additionally \(U^N\) satisfies \cref{eq:U^N-bound}
and \(m^N_0\!\big/ m^N_\infty\) satisfies \cref{eq:h^N_0-bound} without loss of generality,
as they can be removed by the argument in the end of the proof of \cref{lem:ps-unif-hypocoer}.
We define
\begin{multline*}
\mathcal E^N (t,m^N) = F^N(m^N) + \frac 12\int |\mathbf v^2| m^N + H(m^N)
+ a t \int\bigl|\nabla_{\mathbf v} \log h^N\bigr|^2 m^N \\
+ 2b t^2 \int \nabla_{\mathbf x} \log h^N\cdot \nabla_{\mathbf v} \log h^N m^N
+ c t^3 \int\bigl|\nabla_{\mathbf x} \log h^N\bigr|^2 m^N
\end{multline*}
where \(h^N \coloneqq m^N\!\big/ m^N_\infty\).
By the computations in \cref{lem:ps-unif-hypocoer}, we have
\(\frac{d}{dt} \mathcal E^N\bigl(t,m^N_t\bigr)
\leq - (Y^N_t)^{\mathsf T} K''_t Y^N_t\),
where \(K''_t\) is equal to
\[
\begin{pmatrix}
1 - a +2at - 2\bigl(M^F_{mx}+M^F_{mm}\bigr) bt^2 & -2bt^2 & -2at -4bt & 0\\
  & 2at & -2 \bigl(M^F_{mx}+M^F_{mm}\bigr) ct^3 & -4bt^2 \\
  & & 2bt^2 - 3ct^2 & 0\\
  & & & 2ct^3
\end{pmatrix}
\]
and \(Y^N_t\) is defined by \cref{eq:Y^N}.
We choose again the constants \(a\), \(b\), \(c\)
depending only on \(M^F_{mx}\) and \(M^F_{mm}\)
such that \(ac > b^2\) and \(t \mapsto \mathcal H^N(t, m^N_t)\) is non-increasing on \([0,1]\).
Hence we have for every \(t \in (0,1]\),
\begin{align*}
I\bigl(m^N_t \big| \hat m^N_\infty\bigr)
&\leq \frac{C}{t^3} \Bigl( \mathcal E^N\bigl(t,m^N_t\bigr)
- \mathcal F^N\bigl(m^N_t\bigr) \Bigr) \\
&\leq \frac{C}{t^3} \Bigl( \mathcal E^N\bigl(t,m^N_t\bigr)
- \mathcal F^N\bigl(m^N_\infty\bigr) \Bigr) \\
&\leq \frac{C}{t^3} \Bigl( \mathcal E^N\bigl(0,m^N_0\bigr)
- \mathcal F^N\bigl(m^N_\infty\bigr) \Bigr) \\
&= \frac{C}{t^3} \Bigl( \mathcal F^N\bigl(m^N_0\bigr)
- \mathcal F^N\bigl(m^N_\infty\bigr) \Bigr) \\
&= \frac{C}{t^3} H\bigl(m^N_0 \big| m^N_\infty\bigr).
\end{align*}
Similarly, we use the fact that
$\mathcal F^N\bigl(m^N_t\bigr) - \mathcal F^N\bigl(m^N_\infty\bigr)
= H \bigl( m^N_t \big| m^N_\infty \bigr) \geq 0$
to get the second inequality.
Here the difference is that the $N$-particle system is linear
and this fact does not rely on the convexity of $F$.
\end{proof}

\subsection{Propagation of chaos}
\label{sec:poc}

Using all the regularization results proved in \cref{sec:log-harnack,sec:herau}
we can finally give the proof of the main theorem.

\begin{proof}[Proof of \cref{thm:poc}]
Let \(m_0\) and \(m^N_0\) be the respective initial values for the dynamics
\cref{eq:mf-fp,eq:ps-fp}
and suppose they have finite second moment.
The first claim of the theorem \cref{eq:w-poc}
can be written as two bounds on \(W_2^2\bigl(m^N_t, m_t^{\otimes N}\bigr)\),
the first of which follows directly from the finite-time bound
in \cref{prop:finite-w-poc}.
The second claim \cref{eq:h-poc} is nothing but \cref{lem:log-harnack-poc}.
It remains to
find some \(C_2\), \(\kappa\) depending only on \(\rho^x\), \(M^F_{mx}\),
\(M^F_{mm}\) and prove
\begin{multline}
\label{eq:large-w-poc}
W_2^2\bigl(m^N_t, m^{\otimes N}_t\bigr)
\leq
\frac{C_2N}{(t \wedge 1)^6} W_2^2(m_0, m_\infty) e^{-\kappa t} \\
+ \frac{C_2}{(t \wedge 1)^6} W_2^2\bigl(m^N_0, m^{\otimes N}_\infty\bigr)
e^{ - ( \kappa - C_2/N ) t }
+ \frac{C_2d}{\kappa - C_2/N}
\end{multline}
for \(t > 0\).
Set \(t_1 = \frac{t \wedge 1}{2}\) and \(t_2 = t \wedge 1\).
By the Wasserstein to entropy regularization result in \cref{prop:w-to-h},
we can find a constant \(C\)
depending on \(M^F_{mx}\) and \(M^F_{mm}\) such that
\[
H(m_{t_1} | m_\infty) \leq \frac{C}{t_1^3} W_2^2(m_0, m_\infty)
\quad\text{and}\quad
H\bigl(m^N_{t_1} \big| m^N_\infty\bigr)
\leq \frac{C}{t_1^3} W_2^2\bigl(m^N_0, m^N_\infty\bigr).
\]
In the following \(C\) may change from line to line
and may depend additionally on the LSI constant \(\rho\).
Applying the regularization in \cref{prop:h-to-i} to the dynamics
with \(m_{t_1}\) and \(m^N_{t_1}\) as respective initial values
and noting that \(t_2 - t_1 \leq 1\) by definition,
we obtain
\begin{align*}
I(m_{t_2} | \hat m_{t_2}) &\leq \frac{C}{(t_2 - t_1)^3} \bigl(\mathcal F(m_{t_1}) - \mathcal F(m_\infty) \bigr), \\
I\bigl(m^N_{t_2}\big|\hat m^N_{t_2}\bigr)
&\leq \frac{C}{(t_2 - t_1)^3} H\bigl(m^N_{t_1}\big| m^N_\infty\bigr),
\end{align*}
whereas \( \mathcal F(m_{t_1}) - \mathcal F(m_\infty) \)
is bounded by the entropy sandwich in \cref{lem:mf-entropy-sandwich}:
\[
\mathcal F(m_{t_1}) - \mathcal F(m_\infty) \leq C H(m_{t_1} | m_\infty).
\]
Consequently, both the measures \(m_{t_2}\) and \(m^N_{t_2}\) have finite entropy and finite Fisher information,
and we can apply respectively \cref{thm:mf-entropy-convergence,thm:ps-entropy-convergence}
to the dynamics with initial values \(m_{t_2}\) and \(m^N_{t_2}\).
We then obtain
\begin{align*}
\mathcal F(m_t) - \mathcal F(m_\infty)
&\leq \frac{C}{t_1^6} W_2^2(m_0, m_\infty) e^{-\kappa (t - t_2)} \\
\intertext{and}
\mathcal F^N\bigl(m^N_t\bigr) - N\mathcal F(m_\infty)
&\leq \frac{C}{t_1^6} W_2^2\bigl(m^N_0, m_\infty^{\otimes N}\bigr)
e^{-(\kappa - C/N) (t - t_2)}
+ \frac{Cd}{\kappa - C/N}.
\end{align*}
Using consecutively the triangle inequality,
Talagrand's inequality \cref{eq:t2} for \(m_\infty^{\otimes N}\)
and the entropy inequalities in \cref{lem:mf-entropy-sandwich,lem:ps-entropy-sandwich},
we have
\begin{align*}
W_2^2\bigl(m^N_t, m^{\otimes N}_t\bigr)
&\leq 2 W_2^2\bigl(m^N_t, m^{\otimes N}_\infty\bigr) + 2 N W_2^2 (m_t, m_\infty) \\
&\leq \frac{2}{\rho}
\Bigl(H\bigl(m^N_t\big| m^{\otimes N}_\infty\bigr)
+ NH(m_t | m_\infty)\Bigr) \\
&\leq \frac{2}{\rho}
\Bigl(\mathcal F^N\bigl(m^N_t\bigr) - N\mathcal F(m_\infty)
+ N\bigl(\mathcal F(m_t) - \mathcal F(m_\infty)\bigr)\Bigr)
\end{align*}
So the inequality \cref{eq:large-w-poc} is proved by combining the above
three inequalities.
\end{proof}

\appendix

\section{Lower-semicontinuities}
\label{sec:lower-semi-cont}

\begin{lem}
\label{lem:lower-semi-cont}
The entropy \(H(\cdot)\) and the Fisher information \(I(\cdot)\)
are lower-semicontinuous with respect to the weak topology of \(\mathcal P_2\).
Consequently,
under the assumption \cref{eq:lip},
if \((m_n)_{n \in \mathbb N}\) is a sequence converging to \(m_*\) in \(\mathcal P_2(\mathbb R^{2d})\),
then
\[
\liminf_{n \to +\infty} H(m_n | \hat m_n) \ge H(m_* | \hat m_*)
\quad\text{and}\quad
\liminf_{n \to +\infty} I(m_n | \hat m_n) \ge  I(m_* | \hat m_*).
\]
\end{lem}

\begin{proof}
The lower semicontinuity of \(m \mapsto H(m)\) is classical.
We show the lower semicontinuity of the Fisher information.
Let \((m_n)_{n \in \mathbb N}\) be a sequence
converging to \(m_*\) in \(\mathcal P_2(\mathbb R^d)\).
Without loss of generality, we suppose \(I(m_n) \leq M^2\) for every \(n \in \mathbb N\).
This implies in particular \(\Vert \nabla m_n \Vert_{L^1} \leq M\) by Cauchy--Schwarz.
For every function \(\varphi\) belonging to \(C^\infty_c(\mathbb R^d)\),
we have
\[
\lvert\langle \nabla \varphi, m_* \rangle\rvert
= \lim_{n \to +\infty} \lvert\langle \nabla \varphi, m_n \rangle\rvert
\leq M \Vert \varphi \Vert_{\infty}.
\]
Hence \(\Vert \nabla m_* \Vert_\textnormal{TV} \leq M\) as well.
Moreover, for every \(f \in C_c(\mathbb R^d)\) and every \(\varepsilon > 0\),
we can find \(\varphi \in C^{\infty}_c(\mathbb R^d)\) such that
\(\Vert f - \varphi \Vert_\infty < \frac{\varepsilon}{4M}\)
and \(n \in \mathbb N\) such that
\(| \langle \nabla \varphi, m_n - m_* \rangle | < \frac\varepsilon 2\).
Then,
\begin{multline*}
\lvert\langle f, \nabla (m_n - m_*)\rangle\rvert
\leq \lvert\langle f - \varphi, \nabla m_n \rangle\rvert
+ \lvert\langle f - \varphi, \nabla m_* \rangle\rvert
+ \lvert\langle\nabla\varphi, m_n - m_* \rangle\rvert \\
< 2 \cdot \frac{\varepsilon}{4M} \cdot M + \frac{\varepsilon}{2} = \varepsilon.
\end{multline*}
Equivalently, the sequence of \(\mathbb R^d\)-valued Radon measures \((\nabla m_n)_{n \in \mathbb N}\) converges to \(\nabla m\) \emph{locally} weakly.
We then apply \cite[Theorem 2.34]{ambrosio2000functions} to obtain
\[
\liminf_{n \to +\infty} I(m_n) \geq I(m_*).
\]

Finally, the lower semicontinuity of \(m \mapsto H(m | \hat m)\) (resp.\ \(m \mapsto I(m | \hat m)\))
follows from the lower semicontinuity of \(m \mapsto H(m)\) (resp.\ \(m \mapsto I(m)\))
and the locally uniform quadratic growth of \(x \mapsto \frac{\delta F}{\delta m}(m,x)\)
(resp.\ the locally uniform linear growth of \(x \mapsto D_m F(m,x)\)).
\end{proof}

\section{Convergence of non-linear functional of empirical measures}
\label{sec:convergence-empirical}

Let \(\phi : \mathcal P_2(\mathbb R^d) \to \mathbb R\) be a (non-linear) mean field functional
and \(m\) be a probability measure with finite second moment.
We suppose the first and second-order functional derivatives
\(\frac{\delta \phi}{\delta m}, \frac{\delta^2 \phi}{\delta m^2}\) exist
and that \((\phi, m)\) satisfies
\begin{gather}
\forall m' \in \mathcal P_2(\mathbb R^d),~\forall x \in \mathbb R^d,\quad
\bigl|D_m \phi(m',x)\bigr| \leq M_1,
\label{eq:phi-first} \\
\forall m' \in \mathcal P_2(\mathbb R^d),\quad
\biggl| \iint
\biggl[ \frac{\delta^2 \phi}{\delta m^2} (m', x, x) -
\frac{\delta^2 \phi}{\delta m^2} (m', x, y) \biggr] m(dx) m(dy)
\biggr| \leq M_2
\label{eq:phi-second}
\end{gather}
for some constants \(M_1\) and \(M_2\).

\begin{rem}
The condition \cref{eq:phi-second} is a modified version of
the condition \cite[($p$-LFD)]{tse2019quantitative}.
Our version has the advantage of being \emph{intrinsic}:
the left hand side of \cref{eq:phi-second} stays invariant under the change
\(\frac{\delta^2 \phi}{\delta m^2} (m,x,y)
\to \frac{\delta^2\phi}{\delta m^2} (m,x,y)
+ \frac{\delta \phi_1}{\delta m} (m,y)
+ \phi_2(m)\)
for regular enough \(\phi_1\) and \(\phi_2\).
\end{rem}

\begin{lem}
\label{lem:convergence-empirical}
If the mean field functional \(\phi\) and the measure \(m\) satisfy \cref{eq:phi-first,eq:phi-second},
then for \(N\) i.i.d.\ random variables \(\xi_1\), \(\ldots\,\), \(\xi_N \sim \mu\), we have
\begin{equation}
\label{eq:convergence-empirical}
\Expect \bigl[ \lvert\phi (\mu_{\boldsymbol \xi}) - \phi(m) \rvert^2 \bigr] \leq \frac{M_1^2\Var m}{N} + \frac{M_2^2}{4N^2}.
\end{equation}
\end{lem}

\begin{proof}
We have the decomposition
\[
\Expect \bigl[ |\phi (\mu_{\boldsymbol \xi}) - \phi(m)|^2 \bigr]
= \Var \phi (\mu_{\boldsymbol \xi}) + \bigl( \Expect \bigl[\phi(\mu_{\boldsymbol \xi})\bigr] - \phi(m) \bigr)^2.
\]
Thanks to \cref{eq:phi-first}, the mapping \(\xi^i \mapsto \phi(\mu_{\boldsymbol \xi})\) is \(\frac{M_1}{N}\)-Lipschitz continuous,
so by the Efron--Stein inequality we have
\[
\Var \phi(\mu_{\boldsymbol \xi}) \leq \frac{M_1^2}{N} \Var m.
\]
For the second term we apply the argument of \cite[Theorem 4.2.9 (i)]{tse2019quantitative}
and obtain
\[
\bigl|\Expect \bigl[\phi(\mu_{\boldsymbol \xi}) \bigr] - \phi(m) \bigr| \leq \frac{M_2}{2N}. \qedhere
\]
\end{proof}

\section{Validity of Girsanov transforms}
\label{sec:girsanov-ui}

We prove a lemma similar to \cite[Lemma A.1]{HRSS19} which allows us to justify Girsanov transforms.

\begin{lem}
\label{lem:girsanov-ui}
Let \((\Omega, \mathcal F, (\mathcal F_t)_{t \in [0,T]}, \mathbb P)\) be a filtered probability space.
If \(\beta\), \(\gamma\), \(X\), \(Y: \Omega \times [0,T] \to \mathbb R^d\)
are \(\mathcal F_t\)-adapted continuous stochastic processes satisfying
\[
|\beta_t| + |\gamma_t| \leq C (1 + |X_t| + |V_t| )
\]
almost surely for some constant \(C\),
and if the tuple \((X,V,\beta)\) solves
\begin{align*}
dX_t &= V_t dt, \\
dV_t &= \beta_t dt + \sqrt 2 dW_t
\end{align*}
for an \(\mathcal F_t\)-adapted Brownian \(W_t\)
with \(\Expect \bigl[|X_0|^2 + |V_0|^2\bigr] < +\infty\),
then the exponential local martingale
\[
R_\cdot \coloneqq \exp \biggl( \int_0^\cdot \gamma_s \cdot dW_s
- \frac 12 \int_0^\cdot |\gamma_s|^2 ds \biggr)
\]
is uniformly integrable.
\end{lem}

\begin{proof}
It suffices to verify \(\Expect [ R_T ] = 1\).
Put \(M_\cdot = R_\cdot (1 + |X_\cdot|^2 + |V_\cdot|^2 )\).
By Itō's formula, the local semimartingale satisfies
\[
dM_t = 2 R_t \bigl(X_t \cdot Y_t + (\beta_t + \sqrt 2 \gamma_t) \cdot Y_t + 1 \bigr) dt
+ R_t \bigl((1 + |X_t|^2 + |Y_t|^2) \gamma_t + 2 \sqrt 2 Y_t\bigl) \cdot dW_t.
\]
Using the uniform linear growth condition of \(\beta\), \(\gamma\),
we can find a constant \(C\) such that \(t \mapsto e^{-Ct} M_t\)
is a local supermartingale.
But \(e^{-Ct} M_t \geq 0\).
So by Fatou's lemma \(t \mapsto e^{-Ct} M_t\) is really a supermartingale
and this yields \(\Expect [M_t] \leq e^{Ct} \Expect[M_0]\).
The Itō's formula for \(R_\cdot\) writes
\[
dR_t = R_t \gamma_t \cdot dW_t.
\]
So for \(\varepsilon > 0\) the bounded supermartingale \(\frac{R_\cdot}{1 + \varepsilon R_\cdot}\) satisfies
\[
d \frac{R_t}{1 + \varepsilon R_t} = - \frac{\varepsilon R_t^2 \gamma_t^2}{(1 + \varepsilon R_t)^3} dt + \frac{R_t}{(1 + \varepsilon R_t)^2} \gamma_t \cdot dW_t.
\]
Taking expectations on both sides, we obtain
\[
\Expect \biggl[ \frac{R_T}{1 + \varepsilon R_T} \biggr]
= \frac{1}{1+\varepsilon} - \Expect \biggl[\int_0^T \frac{\varepsilon R_t^2 \gamma_t^2 }{(1 + \varepsilon R_t)^3} dt \biggr].
\]
Using the bound \(\frac{\varepsilon R_t^2 \gamma_t^2}{(1 + \varepsilon R_t)^3} \leq R_t \gamma_t^2 \leq CM_t\),
we take the limit \(\varepsilon \to 0\) by the dominated convergence theorem
and obtain \(\Expect [R_T] = 1\).
\end{proof}

\bibliographystyle{plain}
\bibliography{ref}

\end{document}